\documentclass[12pt]{amsart}
\usepackage{mathrsfs}
\usepackage{amssymb} 
\usepackage{url}
\usepackage{cite}
\usepackage{hyperref}
\usepackage{color}
\definecolor{refcolor}{RGB}{0,0,190}
\hypersetup{
    colorlinks,
    citecolor=refcolor,
    filecolor=refcolor,
    linkcolor=refcolor,
    urlcolor=refcolor
}
\usepackage{graphicx} 
\usepackage{amsaddr}

\theoremstyle{definition}
\newtheorem{theorem}{Theorem}[section]
\newtheorem{definition}[theorem]{Definition}
\newtheorem{proposition}[theorem]{Proposition}
\newtheorem{lemma}[theorem]{Lemma}
\newtheorem{corollary}[theorem]{Corollary}
\newtheorem{remark}[theorem]{Remark}
\newtheorem{example}[theorem]{Example}
\newtheorem{counterexample}{Counterexample}

\def\({\left(}
\def\){\right)}

\newcommand{\image}[3]{\begin{figure*}[ht]
\includegraphics[width=#2\textwidth]{#1}
\caption{\small{\label{#1}#3}}\end{figure*}}

\newcommand{\R}{\mathbb{R}}
\newcommand{\N}{\mathbb{N}}

\newcommand{\de}{\textnormal{d}}

\newcommand{\tn}{\textnormal}
\newcommand{\ds}{\displaystyle}

\newcommand{\ie}{\textit{i.e.} }
\newcommand{\cf}{\textit{cf.} }
\newcommand{\eg}{\textit{e.g.} }
\newcommand{\etc}{\textit{etc.}}

\newcommand{\citep}[2]{\cite{#1}, p. #2}

\newcommand{\cfeg}[2]{(\cf \eg \citep{#1}{#2})}

\newcommand{\seepeg}[2]{(see \eg \citep{#1}{#2})}

\newcommand{\rank}{\textnormal{rank }}
\newcommand{\diag}{\textnormal{diag}}

\newcommand{\mf}[1]{\mathfrak{#1}}
\newcommand{\mc}[1]{\mathcal{#1}}
\newcommand{\ms}[1]{\mathscr{#1}}

\newcommand{\IM}{\tn{im }}
\newcommand{\tensors}[3]{\mc T{}^{#1}_{#2}#3}

\newcommand{\sseuclid}[3]{\R^{#1,#2,#3}}

\newcommand{\sref}[1]{\S\ref{#1}}

\newcommand{\idxannih}[2]{#1{}^{#2}{}}
\newcommand{\idxcoannih}[2]{#1{}_{#2}{}}

\newcommand{\radix}[1]{\idxcoannih{#1}{\circ}}
\newcommand{\annih}[1]{\idxannih{#1}{\bullet}}
\newcommand{\coannih}[1]{\idxcoannih{#1}{\bullet}}

\newcommand{\annihg}{\coannih{g}}
\newcommand{\coannihg}{\annih{g}}
\newcommand{\metric}[1]{\langle#1\rangle}
\newcommand{\annihprod}[1]{\coannih{\langle\!\langle#1\rangle\!\rangle}}

\newcommand{\cocontr}{{{}_\bullet}}

\newcommand{\vectmodule}{\mf X}
\newcommand{\fivect}[1]{\vectmodule(#1)}
\newcommand{\fivectnull}[1]{\vectmodule_\circ(#1)}
\newcommand{\fiscal}[1]{\ms F(#1)}
\newcommand{\fiscalannih}[1]{\annih{\ms F}(#1)}
\newcommand{\fiformk}[2]{\mc A^{#1}(#2)}

\newcommand{\annihforms}[1]{\annih{\mc A}(#1)}
\newcommand{\annihformsk}[2]{\annih{\mc A}{}^{#1}(#2)}
\newcommand{\discformsk}[2]{{\mc A}_d{}^{#1}(#2)}

\newcommand{\srformsk}[2]{\annih{\ms A}{}^{#1}(#2)}
\newcommand{\mansigvar}[1]{#1{}_{\wr}}

\newcommand{\lie}{\mc L}
\newcommand{\kosz}{\mc K}
\newcommand{\der}{\nabla}
\newcommand{\dera}[1]{\der_{#1}}
\newcommand{\derb}[2]{\dera{#1}{#2}}
\newcommand{\derc}[3]{({\derb{#1}{#2}})(#3)}
\newcommand{\lder}{\der^{\flat}}
\newcommand{\ldera}[1]{\lder_{#1}}
\newcommand{\lderb}[2]{\ldera{#1}{#2}}
\newcommand{\lderc}[3]{(\lderb{#1}{#2})(#3)}
\newcommand{\curv}[2]{\mc R^\flat_{#1#2}}

\newcommand{\ric}{\tn{Ric}}

\newcommand{\cyclic}{\circlearrowleft}

\newcommand{\abs}[1]{\left|#1\right|}
\newcommand{\dsfrac}[2]{\ds{\frac{#1}{#2}}}

\newcommand{\flrw}{Friedmann-Lema\^itre-Robertson-Walker}
\newcommand{\FLRW}{FLRW}
\newcommand{\schw}{Schwarzschild}
\newcommand{\rn}{Reissner-Nordstr\"om}
\newcommand{\kn}{Kerr-Newman}

\def\hyph{-\penalty0\hskip0pt\relax}
\newcommand{\semiriem}{semi{\hyph}Riemannian}

\newcommand{\semieucl}{semi{\hyph}Euclidean}
\newcommand{\ssemieucl}{Semi{\hyph}Euclidean}
\newcommand{\semireg}{semi{\hyph}regular}
\newcommand{\ssemireg}{Semi{\hyph}regular}
\newcommand{\nondeg}{non{\hyph}degenerate}

\newcommand{\rstationary}{radical{\hyph}stationary}
\newcommand{\rrstationary}{Radical{\hyph}stationary}
\newcommand{\rannih}{radical{\hyph}annihilator}

\makeatletter
\renewcommand\section{\@startsection {section}{1}{\z@}%
                                   {-3.5ex \@plus -1ex \@minus -.2ex}%
                                   {2.3ex \@plus.2ex}%
                                   {\center\normalfont\Large\bfseries}}
\renewcommand\subsection{\@startsection {subsection}{2}{\z@}%
                                   {-3.5ex \@plus -1ex \@minus -.2ex}%
                                   {2.3ex \@plus.2ex}%
                                   {\normalfont\large\bfseries}}
\renewcommand\subsubsection{\@startsection {subsubsection}{3}{\z@}%
                                   {-3.5ex \@plus -1ex \@minus -.2ex}%
                                   {2.3ex \@plus.2ex}%
                                   {\normalfont\bfseries}}
\makeatother

\begin{document}

\title{On Singular Semi-Riemannian Manifolds}
\author{Ovidiu Cristinel Stoica}
\address{Institute of Mathematics of the Romanian Academy, Bucharest, Romania.\\
Present address: Department of Theoretical Physics, \\National Institute of Physics and Nuclear Engineering -- Horia Hulubei, Bucharest, Romania.}
\email{cristi.stoica@theory.nipne.ro}

\begin{abstract}
On a Riemannian or a {\semiriem} manifold, the metric determines invariants like the Levi-Civita connection and the Riemann curvature. If the metric becomes degenerate (as in singular {\semiriem} geometry), these constructions no longer work, because they are based on the inverse of the metric, and on related operations like the contraction between covariant indices.

In this article we develop the geometry of singular {\semiriem} manifolds. First, we introduce an invariant and canonical contraction between covariant indices, applicable even for degenerate metrics. This contraction applies to a special type of tensor fields, which are {\rannih} in the contracted indices. Then, we use this contraction and the Koszul form to define the covariant derivative for {\rannih} indices of covariant tensor fields, on a class of singular {\semiriem} manifolds named {\rstationary}. We use this covariant derivative to construct the Riemann curvature, and show that on a class of singular {\semiriem} manifolds, named {\semireg}, the Riemann curvature is smooth.

We apply these results to construct a version of Einstein's tensor whose density of weight 2 remains smooth even in the presence of {\semireg} singularities. We can thus write a densitized version of Einstein's equation, which is smooth, and which is equivalent to the standard Einstein equation if the metric is {\nondeg}.

\bigskip
\noindent 
\keywords{singular semi-Riemannian manifolds,singular semi-Riemannian geometry,degenerate manifolds,semi-regular semi-Riemannian manifolds,semi-regular semi-Riemannian geometry,Einstein equation,singularity theorem}
\end{abstract}


\maketitle

\setcounter{tocdepth}{1}
\tableofcontents


\section{Introduction}

\subsection{Motivation and related advances}

Let $M$ be a differentiable manifold with a symmetric inner product structure, named metric, on its tangent bundle. If the metric is {\nondeg}, we can construct in a canonical way a Levi-Civita connection and the Riemann, Ricci and scalar curvatures. If the metric is allowed to be degenerate (hence $M$ is a singular {\semiriem} manifold), some obstructions prevented the construction of such invariants.

Degenerate metrics are useful because they can arise in various contexts in which {\semiriem} manifolds are used. They are encountered even in manifolds with {\nondeg} (but indefinite) metric, because the metric induced on a submanifold can be degenerate. The properties of such submanifolds were studied \eg in \cite{Kup87a,Kup87c}, \cite{Bej95,Bej96}.

In General Relativity, there are models or situations when the metric becomes degenerate or changes its signature. As the Penrose and Hawking \textit{singularity theorems} \cite{Pen65,Haw66i,Haw66ii,Haw67iii,HP70,HE95} show, Einstein's equation leads to singularities under very general conditions, apparently similar to the matter distribution in our Universe. Therefore, many attempts were done to deal with such singularities. For example it was suggested that Ashtekar's method of ``new variables'' \cite{ASH87,ASH91,Rom93a} can be used to pass beyond the singularities, because the variable $\widetilde E^a_i$ -- a densitized frame of vector fields -- defines the metric, which can be degenerate. Unfortunately, it turned out that in this case the connection variable $A_a^i$ may become singular \cf \eg \cite{Yon97}.

In some cosmological models the initial singularity of the Big Bang is eliminated by making the metric Riemannian for the early Universe. The metric changes the signature when traversing a hypersurface, becoming Lorentzian, so that time emerges from a space dimension. Some particular junction conditions were studied (see \cite{Sak84},\cite{Ellis92a,Ellis92b},\cite{Hay92,Hay93,Hay95}, \cite{Der93}, \cite{Dray91,Dray93,Dray94,Dray95,Dray96,Dray01}, \cite{Koss85,Koss87,Koss93a,Koss93b,Koss94a,Koss94b} \etc). 

Other situation where the metric can become degenerate was proposed by Einstein and Rosen, as a model of charged particles \cite{ER35}.

All these applications in Geometry and General Relativity demand a generalization of the standard methods of {\semiriem} Geometry, to cover the degenerate case. A degenerate metric prevents the standard constructions like covariant derivative and curvature. Manifolds endowed with degenerate metrics were studied by Moisil \cite{Moi40}, Strubecker \cite{Str41,Str42a,Str42b,Str45}, Vr\u{a}nceanu \cite{Vra42}. Notable is the work of Kupeli \cite{Kup87a,Kup87b,Kup87c}, which is limited to the constant signature case.

\subsection{Presentation of this article}

The purpose of this article is twofold:
\begin{enumerate}
	\item to provide a toolbox of geometric invariants, which extend the standard constructions from {\semiriem} geometry to the {\nondeg} case, with constant or variable signature,
	\item and to apply these constructions to extend Einstein's equation to a class of singular spacetimes.
\end{enumerate}

The first goal of this article is to construct canonical invariants such as the covariant derivative and Riemann curvature tensor, in the case of singular {\semiriem} geometry.  The main obstruction for this is the fact that when the metric is degenerate, it doesn't admit an inverse. This prohibits operations like index raising and contractions between covariant indices. This prevents the definition of a Levi-Civita connection, and by this, the construction of the curvature invariants. This article presents a way to construct such invariants even if the metric is degenerate, for a class of singular {\semiriem} manifolds which are named \textit{{\semireg}}.

The second goal is to apply the tools developed here to write a densitized version of Einstein's tensor which remains smooth in the presence of singularities, if the spacetime is {\semireg}. Consequently, we can write a version of Einstein's equation which is equivalent to the standard one if the metric is {\nondeg}. This allows us to extend smoothly the equations of General Relativity beyond the apparent limits imposed by the singularity theorems of Penrose and Hawking  \cite{Pen65,Haw66i,Haw66ii,Haw67iii,HP70,HE95}.

Section \sref{s_singular_semi_riemannian} contains generalities on singular {\semiriem} manifolds, in particular the radical bundle associated to the metric, made of the degenerate tangent vectors. In section \sref{s_dual_inner_prod} are studied the properties of the {\rannih} bundle, consisting in the covectors annihilating the degenerate vectors. Tensor fields which are {\rannih} in some of their covariant indices are introduced. On this bundle we can define a metric which is the next best thing to the inverse of the metric, and which will be used to perform contractions between covariant indices. Section \sref{s_tensors_contraction_sign_const} shows how we can contract covariant indices of tensor fields, so long as these indices are {\rannih}s.

Normally, the Levi-Civita connection is obtained by raising an index of the right member of the Koszul formula (named here Koszul form), operation which is not available when the metric is degenerate. Section \sref{s_koszul_form} studies the properties of the Koszul form, which are similar to those of the Levi-Civita connection. This allows us to construct in section \sref{s_cov_der} a sort of covariant derivative for vector fields, and in \sref{s_cov_der_covect} a covariant derivative for differential forms.

The notion of {\semireg} {\semiriem} manifold is defined in section \sref{s_riemann_curvature} as a special type of singular {\semiriem} manifold with variable signature on which the lower covariant derivative of any vector field, which is a $1$-form, admits smooth covariant derivatives.

The Riemann curvature tensor is constructed in \sref{s_riemann_curvature} with the help of the Koszul form and of the covariant derivative for differential forms introduced in section \sref{s_cov_der}. For {\semireg} {\semiriem} manifolds, the Riemann curvature tensor is shown to be smooth, and to have the same symmetry properties as in the {\nondeg} case. In addition, it is {\rannih} in all of its indices, this allowing the construction of the Ricci and scalar curvatures. Then, in section \sref{s_riemann_curvature_ii}, the Riemann curvature tensor is expressed directly in terms of the Koszul form, obtaining an useful formula. Then the Riemann curvature is compared with a curvature tensor obtained by Kupeli by other means \cite{Kup87b}.

Section \sref{s_semi_reg_semi_riem_man_example} presents two examples of {\semireg} {\semiriem} manifolds. The first is based on diagonal metrics, and the second on degenerate metrics which are conformal to {\nondeg} metrics.

The final section, \sref{s_einstein_tensor_densitized}, applies the results of this article to General Relativity. This section studies the Einstein's equation on {\semireg} {\semiriem} manifolds. It proposes a densitized version of this equation, which remains smooth on {\semireg} spacetimes, and reduces to the standard Einstein equation if the metric is {\nondeg}.

\section{Singular {\semiriem} manifolds}
\label{s_singular_semi_riemannian}

\subsection{Definition of singular {\semiriem} manifolds}
\label{s_singular_semi_riemannian_def}

\begin{definition}(see \eg \cite{Kup87b}, \citep{Pam03}{265} for comparison)
\label{def_sing_semiRiemm_man}
A \textit{singular {\semiriem} manifold} is a pair $(M,g)$, where $M$ is a differentiable manifold, and $g\in \Gamma(T^*M \odot_M T^*M)$ is a symmetric bilinear form on $M$, named \textit{metric tensor} or \textit{metric}. If the signature of $g$ is fixed, then 
$(M,g)$ is said to be with \textit{constant signature}. If the signature of $g$ is allowed to vary from point to point, $(M,g)$ is said to be with \textit{variable signature}. If $g$ is {\nondeg}, then $(M,g)$ is named \textit{{\semiriem} manifold}. If $g$ is positive definite, $(M,g)$ is named \textit{Riemannian manifold}.
\end{definition}

\begin{example}[Singular {\ssemieucl} Spaces $\sseuclid r s t$, \cf \eg \citep{Pam03}{262}]
\label{ex_sing_semi_euclidean}
Let $r,s,t\in\N$, $n=r+s+t$, We define the singular {\semieucl} space $\sseuclid r s t$ by:
\begin{equation}
	\sseuclid r s t:=(\R^n,\metric{,}),
\end{equation}
where the metric acts on two vector fields $X$, $Y$ on $\R^n$ at a point $p$ on the manifold, in the natural chart, by
\begin{equation}
	\metric{X,Y} = -\sum_{i=r+1}^s X^i Y^i + \sum_{j=r+s+1}^n X^j Y^j.
\end{equation}
If $r=0$ we fall over the {\semieucl} space $\R^n_s:=\sseuclid 0 s t$ (see \eg \citep{ONe83}{58}). If $s=0$ we find the degenerate Euclidean space. If $r=s=0$, then $t=n$ and we recover the Euclidean space $\R^n$ endowed with the natural scalar product.
\end{example}

\begin{definition}
\label{def_signature_change}
Let $p\in M$ be a point of a singular {\semiriem} manifold. We say that \textit{the metric changes its signature} at $p$, if any neighborhood of $p$ contains at least a point $q$ where the metric's signature is different than the metric's signature at $p$.
\end{definition}

\begin{remark}
\label{rem_sign_var_points}
Let $(M,g)$ be a singular {\semiriem} manifold and let $\mansigvar M\subseteq M$ be the set of the points where the metric changes its signature. From Definition \ref{def_signature_change}, the set $\mansigvar M$ is closed. The set $M-\mansigvar M$ is dense in $M$, and open, and it is a union of singular {\semiriem} manifolds with constant signature.
\end{remark}

\begin{example}
We define the following metric on the manifold $\R^n$, $n\in\N$:
\begin{equation}
	\metric{X,Y} = \sum_{i=1}^n f_i X^i Y^i
\end{equation}
where $X$, $Y$ are two vector fields on $\R^n$, and $f_i\in\fiscal{\R^n}$. This metric endows $\R^n$ with a structure of singular {\semiriem} manifold. The metric has constant signature on the regions of the form $\prod_{i=1}^n(a_i,b_i)\subseteq\R^n$, where the intervals $(a_i,b_i)$ have the property that $f_i|_{(a_i,b_i)}$ has constant sign (is either $+1$, or $-1$, or constantly equal to $0$). The metric changes its signature at the points $p=(x_1,\ldots,x_i,\ldots,x_n)\in\R^n$ having the property that for some $i\in\{1,\ldots,n\}$, $x_i$ is on the boundary of the support of the function $f_i$.
\end{example}

\subsection{The radical of a singular {\semiriem} manifold}
\label{s_radix}

\begin{definition}(\cf \eg \citep{Bej95}{1}, \citep{Kup96}{3} and \citep{ONe83}{53})
Let $(V,g)$ be a finite dimensional inner product space, where the inner product $g$ may be degenerate. The totally degenerate space $\radix{V}:=V^\perp$ is named the \textit{radical} of $V$. An inner product $g$ on a vector space $V$ is {\nondeg} if and only if $\radix{V}=\{0\}$.
\end{definition}

\begin{definition}(see \eg \citep{Kup87b}{261}, \citep{Pam03}{263})
We denote by $\radix{T}M$ and we call \textit{the radical of $TM$} the following subset of the tangent bundle: $\radix{T}M=\cup_{p\in M}\radix{(T_pM)}$. We can define vector fields on $M$ valued in $\radix{T}M$, by taking those vector fields $W\in\fivect M$ for which $W_p\in\radix{(T_pM)}$. We denote by $\fivectnull{M}\subseteq\fivect M$ the set of these sections -- they form a vector space over $\R$ and a module over $\fiscal{M}$. $\radix{T}M$ is a vector bundle if and only if the signature of $g$ is constant on all $M$, and in this case, $\radix{T}M$ is a distribution.
\end{definition}

\begin{example}
\label{ex_sing_semi_euclidean_radix}
The radical $\radix{T}\sseuclid r s t$ of the singular {\semieucl} manifold $\sseuclid r s t$ in the Example \ref{ex_sing_semi_euclidean} is spanned at each point $p$ by the tangent vectors $\partial_{ap}$ with $a\leq r$:
\begin{equation}
	\radix{T}\sseuclid r s t = \bigcup_{p\in\sseuclid r s t}\tn{span}({\{(p,\partial_{ap})|\partial_{ap}\in T_p\sseuclid r s t,a\leq r\}}).
\end{equation}
The sections of $\radix{T}\sseuclid r s t$ are therefore given by
\begin{equation}
	\fivectnull{\sseuclid r s t} = \{X\in\fivect{\sseuclid r s t}|X=\sum_{a=1}^r X^a\partial_a\}.
\end{equation}
\end{example}

\section{The {\rannih} inner product space}
\label{s_dual_inner_prod}

Let $(V,g)$ be an inner product vector space. If the inner product $g$ is {\nondeg}, it defines an isomorphism $\flat:V\to V^*$ (see \eg \citep{Gibb06}{15}; \citep{GHLF04}{72}). If $g$ is degenerate, $\flat$ remains a linear morphism, but not an isomorphism. This is why we can no longer define a dual for $g$ on $V^*$ in the usual sense. We will see that we can still define canonically an inner product $\annihg\in\flat(V)^*\odot\flat(V)^*$, and use it to define contraction and index raising in a weaker sense than in the {\nondeg} case. This rather elementary construction can be immediately extended to singular {\semiriem} manifolds. It provides a tool to contract covariant indices and construct the invariants we need.

\subsection{The {\rannih} vector space}
\label{s_rad_annih_space}

This section applies well-known elementary properties of linear algebra, with the purpose is to extend fundamental notions related to the {\nondeg} inner product $g$ on a vector space $V$ induced on the dual space $V^*$ \cfeg{Rom08}{59}, to the case when $g$ is allowed to be degenerate. Let $(V,g)$ be an inner product space over $\R$.

\begin{definition}
\label{def_inner_morphism}
The inner product $g$ defines a vector space morphism, named the \textit{index lowering morphism} $\flat:V\to V^*$, by associating to any $u\in V$ a linear form $\flat(u):V\to \R$ defined by $\flat(u)v:=\metric{u,v}$. Alternatively, we use the notation $u^\flat$ for $\flat(u)$. For reasons which will become apparent, we will also use the notation $\annih u:=u^\flat$.
\end{definition}

\begin{remark}
\label{thm_radix_ker}
It is easy to see that $\radix{V}=\ker\flat$, so $\flat$ is an isomorphism if and only if $g$ is {\nondeg}.
\end{remark}

\begin{definition}
\label{def_radical_annihilator}
The \textit{{\rannih}} vector space $\annih{V}:=\IM\flat\subseteq V^*$ is the space of $1$-forms $\omega$ which can be expressed as $\omega=\annih u$ for some $u$, and they act on $V$ by $\omega(v)=\metric{u,v}$.
\end{definition}

Obviously, in the case when $g$ is {\nondeg}, we have the identification $\annih{V}=V^*$.

\begin{remark}
\label{thm_img_ker_radix}
In other words, $\annih{V}$ is the annihilator of $\radix{V}$. It follows that $\dim\annih{V}+\dim\radix{V}=n$.
\end{remark}

\begin{remark}
Any $u'\in V$ satisfying $\annih{u'}=\omega$ differs from $u$ by $u'-u\in\radix{V}$. Such $1$-forms $\omega\in\annih{V}$ satisfy $\omega|_{\radix{V}} = 0$.
\end{remark}

\begin{definition}
\label{def_co_inner_product}
On the vector space $\annih{V}$ we can define a unique {\nondeg} inner product $\annihg$ by $\annihg(\omega,\tau):=\metric{u,v}$, where $\annih u=\omega$ and $\annih v=\tau$. We alternatively use the notation $\annihprod{\omega,\tau}=\annihg(\omega,\tau)$.
\end{definition}

\begin{proposition}
The inner product $\annihg$ from above is well-defined, being independent on the vectors $u,v$ chosen to represent the $1$-forms $\omega$, $\tau$.
\end{proposition}
\begin{proof}
If $u',v'\in V$ are other vectors satisfying $\annih{u'}=\omega$ and $\annih{v'}=\tau$, then $u'-u\in\radix{V}$ and $v'-v\in\radix{V}$. $\metric{u',v'}=\metric{u,v}+\metric{u'-u,v}+\metric{u,v'-v}+\metric{u'-u,v'-v}=\metric{u,v}$.
\end{proof}

\begin{proposition}
\label{thm_cometric_signature}
The inner product $\annihg$ from above is {\nondeg}, and if $g$ has the signature $(r,s,t)$, then the signature of $\annihg$ is $(0,s,t)$.
\end{proposition}
\begin{proof}
Let's take a basis $(e_a)_{a=1}^n$ in which the inner product is diagonal, with the first $r$ diagonal elements being $0$. We have $\annih{e_a}=0$ for $a\in\{1,\ldots,r\}$, and the $1$-forms $\omega_a:=\annih{e_{r+a}}$ for $a\in\{1,\ldots,s+t\}$ are the generators of $\annih{V}$. They satisfy $\annihprod{\omega_a,\omega_b}=\metric{e_{r+a},e_{r+b}}$. Therefore, $(\omega_a)_{a=1}^{s+t}$ are linear independent and the signature of $\annihg$ is $(0,s,t)$.
\end{proof}

Figure \ref{degenerate-metric} illustrates the various spaces associated with a degenerate inner product space $(V,g)$ and the inner products induced by $g$ on them.

\image{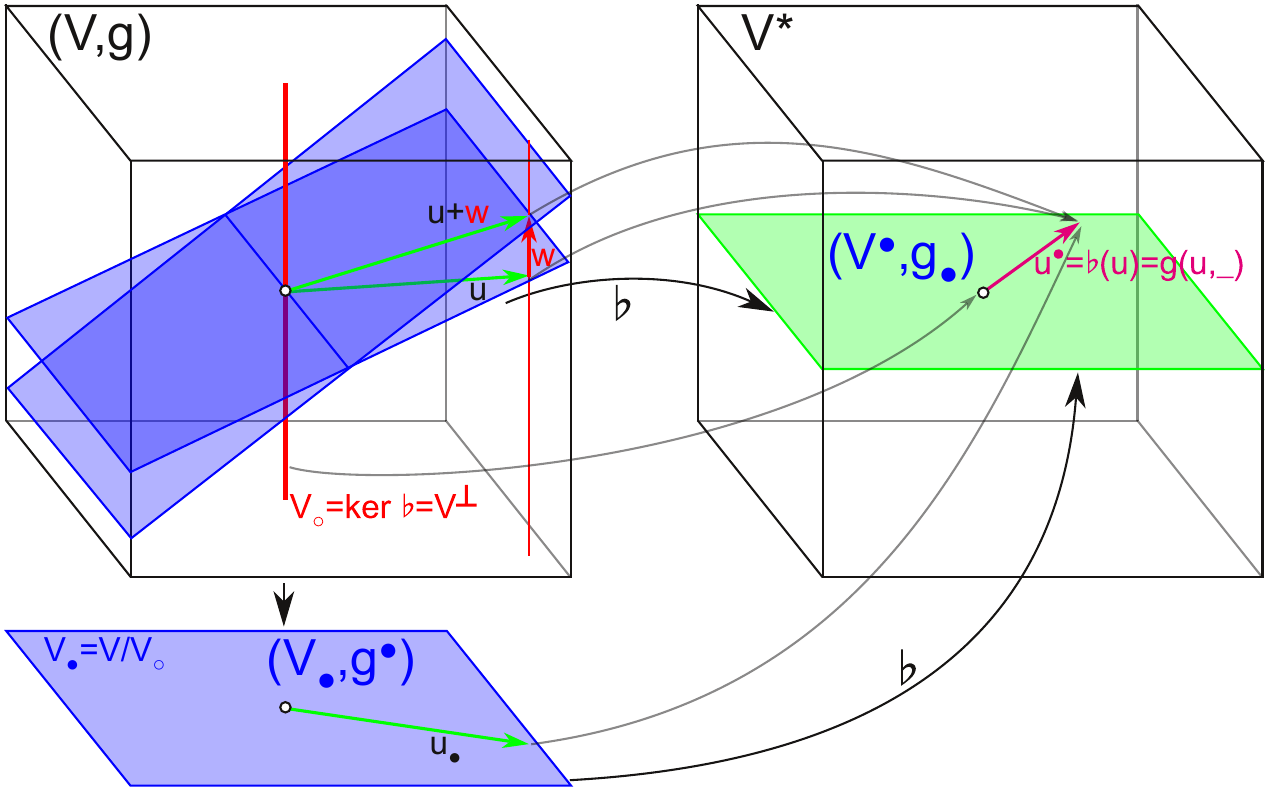}{1.0}{
$(V,g)$ is an inner product vector space. The morphism $\flat:V\to V^*$ is defined by $u\mapsto \annih{u}:=\flat(u)=u^\flat=g(u,\_)$. The radical $\radix{V}:=\ker\flat=V^\perp$ is the set of isotropic vectors in $V$. $\annih{V}:=\IM{\flat}\leq V^*$ is the image of $\flat$. The inner product $g$ induces on $\annih{V}$ an inner product defined by $\annihg(u_1^\flat,u_1^\flat):=g(u_1,u_2)$, which is the inverse of $g$ iff $\det g\neq 0$. The quotient $\coannih{V}:=V/\radix{V}$ consists in the equivalence classes of the form $u+\radix{V}$. On $\coannih{V}$, $g$ induces an inner product $\coannihg(u_1+\radix{V},u_2+\radix{V}):=g(u_1,u_2)$.}

\subsection{The {\rannih} vector bundle}
\label{s_annih}

\begin{definition}
We denote by $\annih{T}M$ the subset of the cotangent bundle defined as
\begin{equation}
	\annih{T}M=\bigcup_{p\in M}\annih{(T_pM)}
\end{equation}
where $\annih{(T_pM)} \subseteq T^*_pM$ is the space of covectors at $p$ which can be expressed as $\omega_p(X_p)=\metric{Y_p,X_p}$ for some $Y_p\in T_p M$ and any $X_p\in T_p M$. $\annih{T}M$ is a vector bundle if and only if the signature of the metric is constant. We can define sections of $\annih{T}M$ in the general case, by
\begin{equation}
	\annihforms{M}:=\{\omega\in\fiformk 1{M}|\omega_p\in\annih{(T_pM)}\tn{ for any }p\in M\}.
\end{equation}
\end{definition}

\begin{remark}
$\annih{(T_pM)}$ is the annihilator space \cfeg{Rom08}{102} of the radical space $\radix{T}_pM$, that is, it contains the linear forms $\omega_p$ which satisfy $\omega_p|_{\radix{T}_pM}=0$.
\end{remark}

\begin{example}
\label{ex_sing_semi_euclidean_annih}
The {\rannih} $\annih{T}\sseuclid r s t$ of the singular {\semieucl} manifold $\sseuclid r s t$ in the Example \ref{ex_sing_semi_euclidean} is:
\begin{equation}
	\annih{T}\sseuclid r s t = \bigcup_{p\in\sseuclid r s t}\tn{span}({\{\de x^a\in T^*_p\sseuclid r s t|a> r\}}).
\end{equation}
Consequently, the {\rannih} $1$-forms have the general form
\begin{equation}
	\omega=\sum_{a=r+1}^n\omega_a\de x^a,
\end{equation}
and
\begin{equation}
	\annihforms{\sseuclid r s t}=\{\omega\in\fiformk 1{\sseuclid r s t}|\omega^i=0,i\leq r\}.
\end{equation}
\end{example}

\subsection{The {\rannih} inner product in a basis}

Let us consider an inner product space $(V,g)$, and a basis $(e_a)_{a=1}^n$ of $V$ in which $g$ takes the diagonal form $g=\diag(\alpha_1,\alpha_2,\ldots,\alpha_n)$, $\alpha_a\in\R$ for all $1\leq a\leq n$. The inner product satisfies:
\begin{equation}
g_{ab}=\metric{e_a,e_b}=\alpha_a\delta_{ab}.
\end{equation}
We also have
\begin{equation*}
\annih{e_a}(e_b):=\metric{e_a,e_b}=\alpha_a\delta_{ab},
\end{equation*}
and, if $(e^{*a})_{a=1}^n$ is the dual basis of $(e_a)_{a=1}^n$,
\begin{equation}
\annih{e_a}=\alpha_a e^{*a}.
\end{equation}

\begin{proposition}
\label{thm_cometric_in_basis}
If in a basis the inner product has the form $g_{ab}=\alpha_a\delta_{ab}$, then 
\begin{equation}
\annihg^{ab}=\frac 1{\alpha_a}\delta^{ab},
\end{equation}for all $a$ so that $\alpha_a\neq 0$.
\end{proposition}
\begin{proof}
Since
\begin{equation*}
\annihprod{\annih{e_a},\annih{e_b}}=\metric{e_a,e_b}=\alpha_a\delta_{ab},
\end{equation*}
and in the same time
\begin{equation*}
\annihprod{\annih{e_a},\annih{e_b}}=\alpha_a \alpha_b \annihprod{e^{*a},e^{*b}}=\alpha_a \alpha_b\annihg^{ab},
\end{equation*}
we have that
\begin{equation*}
\alpha_a \alpha_b\annihg^{ab}=\alpha_a \delta_{ab},
\end{equation*}
This leads, for $\alpha_a\neq 0$, to
\begin{equation*}
\annihg^{ab}=\frac 1{\alpha_a}\delta_{ab}.
\end{equation*}
The case when $\alpha_a = 0$ doesn't happen, since $\annihg$ is defined only on $\IM\flat$.
\end{proof}

\subsection{Radical and {\rannih} tensors}
\label{s_radix_annih_tensors}

For inner product vector spaces we define tensors that are radical in a contravariant slot, and {\rannih} in a covariant slot, and give their characterizations.

\begin{definition}
\label{def_radix_annih_tensor_field}
Let $T$ be a tensor of type $(r,s)$. We call it \textit{radical} in the $k$-th contravariant slot if $T\in \tensors{k-1}0 {M}\otimes_M\radix{T}M\otimes_M \tensors{r-k}s{M}$. We call it \textit{{\rannih}} in the $l$-th covariant slot if  $T\in \tensors r{l-1}{M}\otimes_M\annih{T}M\otimes_M \tensors 0{s-l}{M}$.
\end{definition}

\begin{proposition}
\label{thm_radical_contravariant_index}
A tensor $T\in\tensors r s {M}$ is radical in the $k$-th contravariant slot if and only if its contraction $C^k_{s+1}(T\otimes\omega)$ with any {\rannih} linear $1$-form $\omega\in \fiformk 1{M}$ is zero.
\end{proposition}
\begin{proof}
For simplicity, we can work on an inner product space $(V,g)$ and consider $k=r$ (if $k<r$, we can make use of the permutation automorphisms of the tensor space $\tensors r s V$). 
T can be written as a sum of linear independent terms having the form $\sum_{\alpha}S_{\alpha}\otimes v_{\alpha}$, with $S_{\alpha}\in\tensors{r-1}s V$ and $v_{\alpha}\in V$. We keep only the terms with $S_{\alpha}\neq 0$. The contraction of the $r$-th contravariant slot with any $\omega\in\annih{V}$ becomes $\sum_{\alpha}S_{\alpha}\omega(v_{\alpha})$.

If $T$ is radical in the $r$-th contravariant slot, for all $\alpha$ and any $\omega\in\annih{V}$ we have $\omega(v_{\alpha})=0$, therefore $\sum_{\alpha}S_{\alpha}\omega(v_{\alpha})=0$.

Reciprocally, if $\sum_{\alpha}S_{\alpha}\omega(v_{\alpha})=0$, it follows that for any $\alpha$, $S_{\alpha}\omega(v_{\alpha})=0$. Then, $\omega(v_{\alpha})=0$, because $S_{\alpha}\neq 0$. It follows that $v_{\alpha}\in\radix{V}$.
\end{proof}

\begin{proposition}
\label{thm_radical_annihilator_covariant_index}
A tensor $T\in\tensors r s {M}$ is {\rannih} in the $l$-th covariant slot if and only if its $l$-th contraction with any radical vector field is zero.
\end{proposition}
\begin{proof}
The proof goes as in Proposition \ref{thm_radical_contravariant_index}.
\end{proof}

\begin{example}
\label{thm_metric_radical_annihilator}
The inner product $g$ is {\rannih} in both of its slots. This means that $g\in\annihforms{M}\odot_M\annihforms{M}$.
\end{example}
\begin{proof}
Follows directly from the definition of $\radix{TM}$ and of {\rannih} tensor fields.
\end{proof}

\begin{proposition}
\label{thm_radical_annihilator_vs_radical_contraction}
The contraction between a radical slot and a {\rannih} slot of a tensor is zero.
\end{proposition}
\begin{proof}
Follows from the Proposition \ref{thm_radical_contravariant_index} combined with the commutativity between tensor products and linear combinations with contraction. The proof goes similar to that of the Proposition \ref{thm_radical_contravariant_index}.
\end{proof}

\section{Covariant contraction of tensor fields}
\label{s_tensors_contraction_sign_const}

We don't need an inner product to define contractions between one covariant and one contravariant indices. We can use the inner product $g$ to contract between two contravariant indices, obtaining the \textit{contravariant contraction operator} $C^{kl}$ \cfeg{ONe83}{83}. On the other hand, the contraction is not always well defined for two covariant indices. We will see that we can use $\annihg$ for such contractions, but this works only for vectors or tensors which are {\rannih} in covariant slots. Fortunately, this kind of tensors turn out to be the relevant ones in the applications to singular {\semiriem} geometry.

\subsection{Covariant contraction on inner product spaces}
\label{s_tensors_covariant_contraction_inner_prod}

\begin{definition}
\label{def_contraction_covariant}
We can define uniquely the \textit{covariant contraction} or \textit{covariant trace} operator by the following steps.
\begin{enumerate}
	\item 
We define it first on tensors $T\in\annih{V}\otimes\annih{V}$, by $C_{12}T=\annihg^{ab}T_{ab}$. This definition is independent on the basis, because $\annihg\in\annih{V}^*\otimes\annih{V}^*$. 
	\item 
Let $T\in\tensors r s V$ be a tensor with $r\geq 0$ and $s\geq 2$, which satisfies $T\in V^{\otimes r}\otimes {V^*}^{\otimes {s-2}}\otimes\annih{V}\otimes\annih{V}$, that is, $T(\omega_1,\ldots,\omega_r,v_1,\ldots,v_s)=0$ for any $\omega_i\in V^*, i=1,\ldots,r$, $v_j\in V,j=1,\ldots,s$ whenever $v_{s-1}\in\radix{V}$ or $v_{s}\in\radix{V}$. Then, we define the covariant contraction between the last two covariant slots by the operator
\begin{equation*}
C_{s-1\,s}:=1_{\tensors r {s-2} V}\otimes C_{1,2}:\tensors r {s-2} V\otimes \annih{V}\otimes\annih{V}\to\tensors r {s-2} V,
\end{equation*}
where $1_{\tensors r {s-2} V}:\tensors r {s-2} V\to\tensors r {s-2}V$ is the identity.
In a radical basis, the contraction can be expressed by
\begin{equation*}
\label{eq_contraction_covariant_end}
(C_{s-1\,s} T)^{a_1\ldots a_r}{}_{b_1\ldots b_{s-2}} :=	\annihg^{b_{s-1} b_{s}}T^{a_1\ldots a_r}{}_{b_1\ldots \ldots b_{s-2}b_{s-1}b_{s}}.
\end{equation*}
	\item 
Let $T\in\tensors r s V$ be a tensor with $r\geq 0$ and $s\geq 2$, which satisfies
\begin{equation}
	T\in V^{\otimes r}\otimes {V^*}^{\otimes {k-1}}\otimes\annih{V}\otimes {V^*}^{\otimes l-k-1}\otimes\annih{V}\otimes {V^*}^{\otimes s-l},
\end{equation}
$1\leq k<l\leq s$, that is, $T(\omega_1,\ldots,\omega_r,v_1,\ldots,v_k,\ldots,v_l,\ldots,v_s)=0$ for any $\omega_i\in V^*, i=1,\ldots,r$, $v_j\in V,j=1,\ldots,s$ whenever $v_k\in\radix{V}$ or $v_l\in\radix{V}$. We define the contraction
\begin{equation*}
C_{kl}:V^{\otimes r}\otimes {V^*}^{\otimes {k-1}}\otimes\annih{V}\otimes {V^*}^{\otimes l-k-1}\otimes\annih{V}\otimes {V^*}^{\otimes s-l}
\to V^{\otimes r}\otimes {V^*}^{\otimes {s-2}},
\end{equation*}
by $C_{kl}:=C_{s-1\,s}\circ P_{k,s-1;l,s}$, where $C_{s-1\,s}$ is the contraction defined above, and $P_{k,s-1;l,s}:T\in\tensors r s V\to T\in\tensors r s V$ is the permutation isomorphisms which moves the $k$-th and $l$-th slots in the last two positions. In a basis, the components take the form
\begin{equation}
\label{eq_contraction_covariant_inner_prod_space}
(C_{kl} T)^{a_1\ldots a_r}{}_{b_1\ldots\widehat{b}_k\ldots\widehat{b}_l\ldots b_s} :=	\annihg^{b_k b_l}T^{a_1\ldots a_r}{}_{b_1\ldots b_k\ldots b_l\ldots b_s}.
\end{equation}\end{enumerate}
We denote the contraction $C_{kl} T$ of $T$ also by 
\begin{equation*}
C(T(\omega_1,\ldots,\omega_r,v_1,\ldots,\cocontr,\ldots,\cocontr,\ldots,v_s))
\end{equation*}
or simply
\begin{equation*}
T(\omega_1,\ldots,\omega_r,v_1,\ldots,\cocontr,\ldots,\cocontr,\ldots,v_s).
\end{equation*}
\end{definition}

\subsection{Covariant contraction on singular {\semiriem} manifolds}
\label{ss_tensors_contraction_manifolds}

In \sref{s_tensors_covariant_contraction_inner_prod} we have seen that we can contract in two covariant slots, so long as they are {\rannih}s. 
The covariant contraction uses the inner product $\annihg\in\annih{V}^*\odot\annih{V}^*$. 
In Section \sref{s_radix_annih_tensors} we have extended the notion of tensors which are {\rannih} in some slots to a singular {\semiriem} manifold $(M,g)$ by imposing the condition that the corresponding factors in the tensor product, at $p\in M$, are from $\annih{T}_p M$, which is just a subset of $T^*_p M$. This allows us easily to extend the covariant contraction in {\rannih} slots to singular {\semiriem} manifolds.

\begin{definition}
\label{def_contraction_covariant_ct_sign}
Let $T\in\tensors r s {M}$, $s\geq 2$, be a tensor field on $M$, which is {\rannih} in the $k$-th and $l$-th covariant slots, where $1\leq k<l\leq s$. The \textit{covariant contraction} or \textit{covariant trace} operator is the linear operator
\begin{equation*}
C_{kl}:\tensors r{k-1} {M}\otimes_M\annihforms M\otimes_M\tensors 0{l-k-1} {M}\otimes_M\annihforms M\otimes_M \tensors 0 {s-l}{M} \to \tensors{r}{s-2}{M}
\end{equation*}
by
\begin{equation*}
(C_{kl}T)(p)=C_{kl}(T(p))
\end{equation*}
in terms of the covariant contraction defined for inner product vector spaces, as in \sref{s_tensors_covariant_contraction_inner_prod}.
In local coordinates we have
\begin{equation}
\label{eq_contraction_covariant_ct_sign}
(C_{kl} T)^{a_1\ldots a_r}{}_{b_1\ldots\widehat{b}_k\ldots\widehat{b}_l\ldots b_s} :=	\annihg^{b_k b_l}T^{a_1\ldots a_r}{}_{b_1\ldots b_k\ldots b_l\ldots b_s}.
\end{equation}
We denote the contraction $C_{kl} T$ of $T$ also by 
\begin{equation*}
C(T(\omega_1,\ldots,\omega_r,X_1,\ldots,\cocontr,\ldots,\cocontr,\ldots,X_s))
\end{equation*}
or simply
\begin{equation*}
T(\omega_1,\ldots,\omega_r,X_1,\ldots,\cocontr,\ldots,\cocontr,\ldots,X_s).
\end{equation*}
\end{definition}

\begin{lemma}
\label{thm_contraction_with_metric}
If $T$ is a tensor field $T\in\tensors r s {M}$ with $r\geq 0$ and $s\geq 1$, which is {\rannih} in the $k$-th covariant slot, $1\leq k\leq s$, then its contraction with the metric tensor gives again $T$:
\begin{equation}
\label{eq_contraction_with_metric}
\begin{array}{l}
T(\omega_1,\ldots,\omega_r,X_1,\ldots,\cocontr,\ldots,X_s)\metric{X_k,\cocontr}\\
\,\,\,\,\,=T(\omega_1,\ldots,\omega_r,X_1,\ldots,X_k,\ldots,X_s)
\end{array}
\end{equation}
\end{lemma}
\begin{proof}
For simplicity, we can work on an inner product space $(V,g)$. Let's first consider the case when $T\in\tensors 0 1 V$, in fact, $T=\omega\in\annih V$. Then, equation \eqref{eq_contraction_with_metric} reduces to
\begin{equation}
	\omega(\cocontr)\metric{v,\cocontr}=\omega(v).
\end{equation}
But since $\omega\in\annih V$, it takes the form $\omega=\annih{u}$ for $u\in V$, and $\omega(\cocontr)\metric{v,\cocontr}=\annihprod{\omega,\annih v}=\metric{u,v}=\annih u (v)=\omega(v)$.

The general case is obtained from the linearity of the tensor product in the $k$-th covariant slot.
\end{proof}

\begin{corollary}
\label{thm_contracted_metric_w_metric}
$\metric{X,\cocontr}\metric{Y,\cocontr}=\metric{X,Y}.$
\end{corollary}
\begin{proof}
Follows from Lemma \ref{thm_contraction_with_metric} and from $g\in\annihforms{M}\odot_M\annihforms{M}$.
\end{proof}

\begin{example}
\label{thm_contracted_metric_w_itself}
$\metric{\cocontr,\cocontr}=\rank g.$
\end{example}
\begin{proof}
For simplicity, we can work on an inner product space $(V,g)$. We recall that $g\in\annih{V}\odot\annih{V}$, $\annihg\in\annih{V}^*\odot\annih{V}^*$. When restricted to $\annih{V}$ and $\annih{V}^*$ they are {\nondeg} and inverse to one another. Since $\dim\annih{V}=\dim\ker\flat=\rank g$, we obtain $\metric{\cocontr,\cocontr}=\rank g$.
\end{proof}

\begin{theorem}
\label{thm_contraction_orthogonal}
Let $(M,g)$ be a singular {\semiriem} manifold with constant signature. Let $T\in\tensors r s M$, $s\geq 2$, be a tensor field which is {\rannih} in the $k$-th and $l$-th covariant slots ($1\leq k<l \leq n$). Let $(E_a)_{a=1}^n$ be an orthogonal basis on $M$, so that $E_1,\ldots,E_{n-\rank g}\in\fivectnull{M}$. Then
\begin{equation}
\label{eq_cov_contraction_orthogonal}
\begin{array}{l}
T(\omega_1,\ldots,\omega_r,X_1,\ldots,\cocontr,\ldots,\cocontr,\ldots,X_s) \\
\,\,\,\,\,
=\sum_{a=n-\rank g+1}^n \dsfrac{1}{\metric{E_a,E_a}}T(\omega_1,\ldots,\omega_r,X_1,\ldots,E_a,\ldots,E_a,\ldots,X_s),
\end{array}
\end{equation}
for any $X_1,\ldots,X_s\in\fivect M,\omega_1,\ldots,\omega_r\in\fiformk 1 M$.
\end{theorem}
\begin{proof}
For simplicity, we will work on an inner product space $(V,g)$.
From the Proposition \ref{thm_cometric_in_basis} we recall that $\annihg$ is diagonal and $\annihg^{aa}=\ds{\frac 1{g_{aa}}}$, for $a>n-\rank g$. Therefore
\begin{equation*}
\begin{array}{l}
\annihg^{ab}T(\omega_1,\ldots,\omega_r,v_1,\ldots,E_a,\ldots,E_b,\ldots,v_s) \\
\,\,\,\,\,
=\sum_{a=n-\rank g+1}^n \dsfrac{1}{\metric{E_a,E_a}}T(\omega_1,\ldots,\omega_r,v_1,\ldots,E_a,\ldots,E_a,\ldots,v_s).
\end{array}
\end{equation*}
\end{proof}

\begin{remark}
\label{rem_contraction_orthonormal_invariant}
Since in fact 
\begin{equation}
\label{eq_annihprod_orthogonal}
\annihprod{\omega_1,\omega_2}=\sum_{a=n-\rank g+1}^n \dsfrac{\omega_1(E_a)\omega_2(E_a)}{\metric{E_a,E_a}},
\end{equation}
for any {\rannih} $1$-forms $\omega_1,\omega_2\in\annihforms M$, it follows that if we define the contraction alternatively by the equation \eqref{eq_cov_contraction_orthogonal}, the definition is independent on the frame $(E_a)_{a=1}^n$.
\end{remark}

\begin{remark}
\label{rem_contraction_sign_change}
On regions of constant signature, the covariant contraction of a smooth tensor is smooth. But at the points where the signature changes, the contraction is not necessarily smooth, because the inverse of the metric becomes divergent at the points where the signature changes, as it follows from equation \eqref{thm_cometric_in_basis}. The fact that $\annihg_p\in(\annih{T}_pM)^*\odot(\annih{T}_pM)^*$ raises some problems, because the union of $(\annih{T}_pM)^*$ does not form a bundle, and for $\annihg$ the notions of continuity and smoothness don't even make sense.

\begin{counterexample}
The covariant contraction of the two indices of the metric tensor at a point $p\in M$ is $g_p(\cocontr,\cocontr)=\rank g(p)$ (see Example \ref{thm_contracted_metric_w_itself}). When $\rank g(p)$ is not constant, $g_p(\cocontr,\cocontr)$ is discontinuous.
\end{counterexample}

On the other hand, the following example shows that it is possible to have smooth contractions even when the signature changes:
\end{remark}

\begin{example}
\label{ex_contraction_sign_change_smooth}
If $X\in \fivect{M}$ and $\omega\in\annihforms{M}$, $C_{12}(\omega\otimes_M X^\flat)=\annihprod{\omega,X^\flat}=\omega(X)$ and it is smooth, even if the signature is variable.
\end{example}

\begin{remark}
\label{rem_contraction_sign_change_smooth}
Since the points where the signature doesn't change form a dense subset of $M$ (Remark \ref{rem_sign_var_points}), it makes sense to impose the condition of smoothness of the covariant contraction of a smooth tensor. To check smoothness, we simply check whether the extension by continuity of the contraction is smooth.
\end{remark}

\section{The Koszul form}
\label{s_koszul_form}

For convenience, we name \textit{Koszul form} the right member of the Koszul formula (see \eg \citep{ONe83}{61}):

\begin{definition}[The Koszul form, see \eg \citep{Kup87b}{263}]
\label{def_Koszul_form}
\textit{The Koszul form} is defined as
\begin{equation*}
	\kosz:\fivect M^3\to\R,
\end{equation*}
\begin{equation}
\label{eq_Koszul_form}
\begin{array}{llll}
	\kosz(X,Y,Z) &:=&\ds{\frac 1 2} \{ X \metric{Y,Z} + Y \metric{Z,X} - Z \metric{X,Y} \\
	&&\ - \metric{X,[Y,Z]} + \metric{Y, [Z,X]} + \metric{Z, [X,Y]}\}.
\end{array}
\end{equation}
\end{definition}

The Koszul formula becomes
\begin{equation}
\label{eq_koszul_formula}
	\metric{\derb X Y,Z} = \kosz(X,Y,Z),
\end{equation}
and for {\nondeg} metric, the unique Levi-Civita connection is obtained by raising the $1$-form $\kosz(X,Y,\_)$:
\begin{equation}
\label{eq_koszul_formula_inv}
	\derb X Y = \kosz(X,Y,\_)^\sharp.
\end{equation}
If the metric is degenerate, then this is not in general possible. We can raise $\kosz(X,Y,\_)$ on regions of constant signature, and what we obtain is what Kupeli (\citep{Kup87b}{261--262}) called \textit{Koszul derivative} -- which is in general not a connection and is not unique. Kupeli's construction is done only for singular {\semiriem} manifolds with metrics with constant signature, which satisfy the condition of \textit{radical-stationarity} (Definition \ref{def_radical_stationary_manifold}). But if the metric changes its signature, the Koszul derivative is discontinuous at the points where the signature changes. In this article we would not need to use the Koszul derivative, because for our purpose it will be enough to work with the Koszul form.

\subsection{Basic properties of the Koszul form}
\label{s_koszul_form_props}

Let's recall the Lie derivative of a tensor field $T\in\tensors 0 2 M$:

\begin{definition}\seepeg{HE95}{30}
\label{def_lie_derivative_metric}
Let $M$ be a differentiable manifold. Recall that the \textit{Lie derivative} of a tensor field $T\in\tensors 0 2 M$ with respect to a vector field $Z\in\fivect M$ is given by
\begin{equation}
	(\lie_Z T)(X,Y):=Z T(X,Y) - T([Z,X],Y) - T(X,[Z,Y])
\end{equation}
for any $X,Y\in\fivect M$.
\end{definition}

The following properties of the Koszul form correspond directly to standard properties of the Levi-Civita connection of a {\nondeg} metric \cfeg{ONe83}{61}. We prove them explicitly here, because in the case of degenerate metric the proofs need to avoid using the Levi-Civita connection and the index raising. These properties will turn out to be important for what it follows.

\begin{theorem}
\label{thm_Koszul_form_props}
The Koszul form of a singular {\semiriem} manifold $(M,g)$ has, for any $X,Y,Z\in\fivect M$ and $f\in\fiscal M$, the following properties:
\begin{enumerate}
	\item \label{thm_Koszul_form_props_linear}
	It is additive and $\R$-linear in each of its arguments.
	\item \label{thm_Koszul_form_props_flinearX}
	It is $\fiscal M$-linear in the first argument:

	$\kosz(fX,Y,Z) = f\kosz(X,Y,Z).$
	\item \label{thm_Koszul_form_props_flinearY}
	Satisfies the \textit{Leibniz rule}:

	$\kosz(X,fY,Z) = f\kosz(X,Y,Z) + X(f) \metric{Y,Z}.$
	\item \label{thm_Koszul_form_props_flinearZ}
	It is $\fiscal M$-linear in the third argument:

	$\kosz(X,Y,fZ) = f\kosz(X,Y,Z).$
	\item \label{thm_Koszul_form_props_commutYZ}
	It is \textit{metric}:

	$\kosz(X,Y,Z) + \kosz(X,Z,Y) = X \metric{Y,Z}$.
	\item \label{thm_Koszul_form_props_commutXY}
	It is \textit{symmetric} or \textit{torsionless}:

	$\kosz(X,Y,Z) - \kosz(Y,X,Z) = \metric{[X,Y],Z}$.
	\item \label{thm_Koszul_form_props_commutZX}
	Relation with the Lie derivative of $g$:

	$\kosz(X,Y,Z) + \kosz(Z,Y,X) = (\lie_Y g)(Z,X)$.
	\item \label{thm_Koszul_form_props_commutX2Y}

	$\kosz(X,Y,Z) + \kosz(Y,Z,X) = Y\metric{Z,X} + \metric{[X,Y],Z}$.
	\end{enumerate}
\end{theorem}
\begin{proof}
\eqref{thm_Koszul_form_props_linear}\ 
Follows from Definition \ref{def_Koszul_form}, and from the linearity of $g$, of the action of vector fields on scalars, and of the Lie brackets.

\begin{equation*}
\begin{array}{llll}
\eqref{thm_Koszul_form_props_flinearX}\ 
&2\kosz(fX,Y,Z) &=& fX \metric{Y,Z} + Y \metric{Z,fX} - Z \metric{fX,Y} \\
	&&&- \metric{fX,[Y,Z]} + \metric{Y, [Z,fX]} + \metric{Z, [fX,Y]} \\
	&&=& fX \metric{Y,Z} + Y (f\metric{Z,X}) - Z (f\metric{X,Y}) \\
	&&&- f\metric{X,[Y,Z]}+ \metric{Y, f[Z,X] + Z(f)X} \\
	&&&+ \metric{Z, f[X,Y] - Y(f)X} \\
	&&=& fX \metric{Y,Z} + fY \metric{Z,X} \\
	&&&+ Y(f) \metric{Z,X} - fZ \metric{X,Y} \\
	&&&- Z(f)\metric{X,Y} - f\metric{X,[Y,Z]} + f\metric{Y, [Z,X]} \\
	&&&+ Z(f)\metric{Y, X} + f\metric{Z, [X,Y]} - Y(f)\metric{Z, X} \\
	&&=& fX \metric{Y,Z} + fY \metric{Z,X} - fZ \metric{X,Y} \\
	&&& - f\metric{X,[Y,Z]} + f\metric{Y,[Z,X]} + f\metric{Z,[X,Y]} \\
	&&=& 2f\kosz(X,Y,Z) \\
\end{array}
\end{equation*}

\begin{equation*}
\begin{array}{llll}
\eqref{thm_Koszul_form_props_flinearY}\ 
&2\kosz(X,fY,Z) &=& X \metric{fY,Z} + fY \metric{Z,X} - Z \metric{X,fY} \\
	&&&- \metric{X,[fY,Z]} + \metric{fY, [Z,X]} + \metric{Z, [X,fY]} \\
	&&=& X(f) \metric{Y,Z} + fX \metric{Y,Z} \\
	&&& + fY \metric{Z,X} - Z(f) \metric{X,Y} \\
	&&& -fZ \metric{X,Y}- f\metric{X,[Y,Z]} + Z(f)\metric{X,Y} \\
	&&& +f\metric{Y,[Z,X]} + f\metric{Z,[X,Y]} + X(f)\metric{Z,Y} \\
	&&=& f(X \metric{Y,Z} + Y \metric{Z,X}- Z \metric{X,Y} \\
	&&&- \metric{X,[Y,Z]} + \metric{Z,[X,Y]} + \metric{Y,[Z,X]}) \\
	&&&+ X(f) \(\metric{Y,Z}  + \metric{Z,Y}\) \\
	&&=& 2\(f\kosz(X,Y,Z) + X(f) \metric{Y,Z}\) \\
\end{array}
\end{equation*}

\begin{equation*}
\begin{array}{llll}
\eqref{thm_Koszul_form_props_flinearZ}\ 
&2\kosz(X,Y,fZ) &=& X \metric{Y,fZ} + Y \metric{fZ,X} - fZ \metric{X,Y} \\
	&&&- \metric{X,[Y,fZ]} + \metric{Y, [fZ,X]} + \metric{fZ, [X,Y]} \\
	&&=& fX \metric{Y,Z} + X(f) \metric{Y,Z} \\
	&&& + fY \metric{Z,X}+ Y (f)\metric{Z,X} \\
	&&&- fZ (\metric{X,Y})- f\metric{X,[Y,Z]} - Y(f)\metric{X,Z} \\
	&&&+ f\metric{Y,[Z,X]} - X(f)\metric{Y,Z} + f\metric{Z,[X,Y]} \\
	&&=& fX \metric{Y,Z} + fY \metric{Z,X} - fZ (\metric{X,Y}) \\
	&&&- f\metric{X,[Y,Z]}+ f\metric{Y,[Z,X]} + f\metric{Z,[X,Y]} \\
	&&=& 2f\kosz(X,Y,Z) \\
\end{array}
\end{equation*}

\begin{equation*}
\begin{array}{llll}
\eqref{thm_Koszul_form_props_commutYZ}\ 
&	2[\kosz(X,Y,Z) &+& \kosz(X,Z,Y)] \\
	&&=& X \metric{Y,Z} + Y \metric{Z,X} - Z \metric{X,Y} \\
	&&&- \metric{X,[Y,Z]} + \metric{Y,[Z,X]} + \metric{Z,[X,Y]} \\
	&&&+ X \metric{Z,Y} + Z \metric{Y,X} - Y \metric{X,Z} \\
	&&&- \metric{X,[Z,Y]} + \metric{Z,[Y,X]} + \metric{Y,[X,Z]} \\
	&&=& X \metric{Y,Z} - \metric{X,[Y,Z]} \\
	&&&+ \metric{Y,[Z,X]} + \metric{Z,[X,Y]} + X \metric{Y,Z}  \\
	&&&+ \metric{X,[Y,Z]} - \metric{Z,[X,Y]} - \metric{Y,[Z,X]} \\
	&&=& 2X \metric{Y,Z} \\
\end{array}
\end{equation*}

\begin{equation*}
\begin{array}{llll}
\eqref{thm_Koszul_form_props_commutXY}\ 
& 2[\kosz(X,Y,Z) &-& \kosz(Y,X,Z)] \\
	&&=& X \metric{Y,Z} + Y \metric{Z,X} - Z \metric{X,Y} \\
	&&&- \metric{X,[Y,Z]} + \metric{Y,[Z,X]} + \metric{Z,[X,Y]} \\
	&&&- Y \metric{X,Z} - X \metric{Z,Y} + Z \metric{Y,X} \\
	&&&+ \metric{Y,[X,Z]} - \metric{X,[Z,Y]} - \metric{Z,[Y,X]} \\
	&&=& X \metric{Y,Z} + Y \metric{Z,X} - Z \metric{X,Y} \\
	&&&- \metric{X,[Y,Z]} + \metric{Y,[Z,X]} + \metric{Z,[X,Y]} \\
	&&&- Y \metric{Z,X} - X \metric{Y,Z} + Z \metric{X,Y} \\
	&&&- \metric{Y,[Z,X]} + \metric{X,[Y,Z]} + \metric{Z,[X,Y]} \\
	&&=&2 \metric{Z,[X,Y]} = 2 \metric{[X,Y],Z} \\
\end{array}
\end{equation*}

\begin{equation*}
\begin{array}{llll}
\eqref{thm_Koszul_form_props_commutZX}\ 
& 2[\kosz(X,Y,Z) &+& \kosz(Z,Y,X)] \\
	&&=& X \metric{Y,Z} + Y \metric{Z,X} - Z \metric{X,Y} \\
	&&&- \metric{X,[Y,Z]} + \metric{Y,[Z,X]} + \metric{Z,[X,Y]} \\
	&&& +Z \metric{Y,X} + Y \metric{X,Z} - X \metric{Z,Y} \\
	&&&- \metric{Z,[Y,X]} + \metric{Y,[X,Z]} + \metric{X,[Z,Y]} \\
	&&=& X \metric{Y,Z} + Y \metric{Z,X} - Z \metric{X,Y} \\
	&&&- \metric{X,[Y,Z]} + \metric{Y,[Z,X]} + \metric{Z,[X,Y]} \\
	&&& +Z \metric{X,Y} + Y \metric{Z,X} - X \metric{Y,Z} \\
	&&&+ \metric{Z,[X,Y]} - \metric{Y,[Z,X]} - \metric{X,[Y,Z]} \\
	&&=& 2Y \metric{Z,X} - 2\metric{X,[Y,Z]} + 2\metric{Z,[X,Y]} \\
	&&=& 2(Y \metric{Z,X} - \metric{X,\lie_YZ} - \metric{Z,\lie_YX}) \\
	&&=& 2(\lie_Y g)(Z,X) \\
\end{array}
\end{equation*}

\eqref{thm_Koszul_form_props_commutX2Y} By subtracting \eqref{thm_Koszul_form_props_commutXY} from \eqref{thm_Koszul_form_props_commutYZ}, we obtain
$$\kosz(Y,X,Z) + \kosz(X,Z,Y) = X \metric{Y,Z} - \metric{[X,Y],Z}.$$
By applying the permutation $(X,Y,Z)\mapsto(Y,X,Z)$ we get
$$\kosz(X,Y,Z) + \kosz(Y,Z,X) = Y\metric{Z,X} + \metric{[X,Y],Z}.$$
\end{proof}

\begin{remark}
\label{thm_Koszul_form_index}
If $U\subseteq M$ is an open set in $M$ and $(E_a)_{a=1}^n\subset\fivect U$ are vector fields on $U$ forming a frame of $T_pU$ at each $p\in U$, then
\begin{equation}
\label{eq_Koszul_form_index}
\begin{array}{lll}
	\kosz_{abc}&:=&\kosz(E_a,E_b,E_c) \\
	&=&\ds{\frac 1 2} \{E_a(g_{bc}) + E_b(g_{ca}) - E_c(g_{ab})
	- g_{as} \ms C^s_{bc} + g_{bs} \ms C^s_{ca} + g_{cs} \ms C^s_{ab}\},
\end{array}
\end{equation}
where $g_{ab} = \metric{E_a,E_b}$ and $\ms C^c_{ab}$ are the coefficients of the Lie bracket of vector fields \seepeg{Das07}{107}, $[E_a,E_b] = \ms C_{ab}^c E_c$.

The equations (\ref{thm_Koszul_form_props_commutYZ} -- \ref{thm_Koszul_form_props_commutX2Y}) in Theorem \ref{thm_Koszul_form_props} become in the basis $(E_a)_{a=1}^n$:
\begin{equation*}
\begin{array}{ll}
	(\ref{thm_Koszul_form_props_commutYZ}')
	& \kosz_{abc} + \kosz_{acb} = E_a (g_{bc}). \\
	(\ref{thm_Koszul_form_props_commutZX}')
	& \kosz_{abc} + \kosz_{cba} = (\lie_{E_b} g)_{ca}. \\
	(\ref{thm_Koszul_form_props_commutXY}')
	& \kosz_{abc} - \kosz_{bac} = g_{sc}\ms C^s_{ab}. \\
	(\ref{thm_Koszul_form_props_commutX2Y}')
	& \kosz_{abc} + \kosz_{bca} = E_b(g_{ca}) + g_{sc}\ms C^s_{ab}. \\
\end{array}
\end{equation*}

If $E_a=\partial_a:=\ds{\frac {\partial}{\partial x^a}}$ for all $a\in\{1,\ldots,n\}$ are the partial derivatives in a coordinate system, $[\partial_a,\partial_b]=0$ and the equation \eqref{eq_Koszul_form_index} reduces to
\begin{equation}
\label{eq_Koszul_form_coord}
	\kosz_{abc}=\kosz(\partial_a,\partial_b,\partial_c)=\ds{\frac 1 2} (
	\partial_a g_{bc} + \partial_b g_{ca} - \partial_c g_{ab}),
\end{equation}
which are Christoffel's symbols of the first kind \cfeg{HE95}{40}.
\end{remark}

\begin{corollary}
\label{thm_Koszul_form}
Let $X,Y\in\fivect{M}$ two vector fields. The map $\kosz(X,Y,\_):\fivect{M}\to\R$ defined as
\begin{equation}
	\kosz(X,Y,\_)(Z) := \kosz(X,Y,Z)
\end{equation}
is a differential $1$-form.
\end{corollary}
\begin{proof}
It is a direct consequence of Theorem \ref{thm_Koszul_form_props}, properties \eqref{thm_Koszul_form_props_linear} and \eqref{thm_Koszul_form_props_flinearZ}.
\end{proof}

\begin{corollary}
\label{thm_Koszul_null_props}
If $X,Y\in \fivect M$ and $W\in\fivectnull{M}$, then
\begin{equation}
	\kosz(X,Y,W) = \kosz(Y,X,W) = -\kosz(X,W,Y) = -\kosz(Y,W,X). \\
\end{equation}
\end{corollary}
\begin{proof}
From Theorem \ref{thm_Koszul_form_props}, property \eqref{thm_Koszul_form_props_commutXY},
\begin{equation}
	\kosz(X,Y,W) = \kosz(Y,X,W) + \metric{[X,Y],W} = \kosz(Y,X,W).
\end{equation}
From Theorem \ref{thm_Koszul_form_props}, property \eqref{thm_Koszul_form_props_commutYZ},
\begin{equation}
\kosz(X,Y,W) = -\kosz(X,W,Y) + X\metric{Y,W}= -\kosz(X,W,Y)
\end{equation}
and
\begin{equation}
\kosz(Y,X,W) = -\kosz(Y,W,X).
\end{equation}
\end{proof}

\section{The covariant derivative}
\label{s_cov_der}

\subsection{The lower covariant derivative of vector fields}
\label{s_l_cov_dev}

\begin{definition}[The lower covariant derivative]
\label{def_l_cov_der}
The \textit{lower covariant derivative} of a vector field $Y$ in the direction of a vector field $X$ is the differential $1$-form $\lderb XY \in \fiformk 1{M}$ defined as
\begin{equation}
\label{eq_l_cov_der_vect}
\lderc XYZ := \kosz(X,Y,Z)
\end{equation}
for any $Z\in\fivect{M}$.
The \textit{lower covariant derivative operator} is the operator
\begin{equation}
	\lder:\fivect{M} \times \fivect{M} \to \fiformk 1{M}
\end{equation}
which associates to each $X,Y\in\fivect{M}$ the differential $1$-form $\ldera XY$.
\end{definition}

\begin{remark}
Unlike the case of the covariant derivative defined when the metric is {\nondeg}, the result of applying the lower covariant derivative to a vector field is not another vector field, but a differential $1$-form. When the metric is {\nondeg} the two are equivalent by changing the type of the $1$-form $\lderb XY$ into a vector field $\derb XY=(\lderb XY)^\sharp$. Similar objects mapping vector fields to $1$-forms were used in \eg \citep{Koss85}{464--465}. The lower covariant derivative doesn't require a {\nondeg} metric, and it will be very useful in what follows.
\end{remark}

The following properties correspond to standard properties of the Levi-Civita connection of a {\nondeg} metric \cfeg{ONe83}{61}, and are extended here to the case when the metric can be degenerate.

\begin{theorem}
\label{thm_l_cov_der_props}
The lower covariant derivative operator $\lder$ of vector fields defined on a singular {\semiriem} manifold $(M,g)$ has the following properties:
\begin{enumerate}
	\item \label{thm_l_cov_der_props_linear}
	It is additive and $\R$-linear in both of its arguments.
	\item \label{thm_l_cov_der_props_flinearX}
	It is $\fiscal M$-linear in the first argument:

	$\lderb{fX}Y = f\lderb XY.$
	\item \label{thm_l_cov_der_props_flinearY}
	Satisfies the \textit{Leibniz rule}:

	$\lderb X{fY} = f\lderb XY + X(f) Y^\flat.$
	
	or, explicitly,

	$\lderc X{fY}Z = f\lderc XYZ + X(f) \metric{Y,Z}.$
	\item \label{thm_l_cov_der_props_flinearZ}
	It is \textit{metric}:

	$\lderc XYZ + \lderc XZY = X \metric{Y,Z}$.
	\item \label{thm_l_cov_der_props_commutXY}
	It is \textit{symmetric} or \textit{torsionless}:

	$\lderb XY - \lderb YX = [X,Y]^\flat$
	
	or, explicitly,
	
	$\lderc XYZ - \lderc YXZ = \metric{[X,Y],Z}$.
	\item \label{thm_l_cov_der_props_commutZX}
	Relation with the Lie derivative of $g$:

	$\lderc XYZ + \lderc ZYX = (\lie_Y g)(Z,X)$.
	\item \label{thm_l_cov_der_props_commutX2Y}

	$\lderc XYZ + \lderc YZX = Y\metric{Z,X} + \metric{[X,Y],Z}$.
	\end{enumerate}
	for any $X,Y,Z\in\fivect M$ and $f\in\fiscal M$.
\end{theorem}
\begin{proof}
Follows from the direct application of Theorem \ref{thm_Koszul_form_props}.
\end{proof}

\subsection{{\rrstationary} singular {\semiriem} manifolds}
\label{s_radical_stationary_manifolds}

The {\rstationary} singular {\semiriem} manifolds of constant signature were introduced by Kupeli in \citep{Kup87b}{259--260}, where he called them singular {\semiriem} manifolds. Later, in \cite{Kup96} Definition 3.1.3, he named them ``stationary singular {\semiriem} manifolds''. Here we use the term ``{\rstationary} singular {\semiriem} manifolds'' to avoid possible confusion, since the word ``stationary'' is used in general for manifolds admitting a Killing vector field, and in particular for spacetimes invariant at time translation. Kupeli introduced them to ensure the existence of the Koszul derivative. Our need is different, since we don't rely on Kupeli's Koszul derivative.

\begin{definition}[\cf \cite{Kup96} Definition 3.1.3]
\label{def_radical_stationary_manifold}
A singular {\semiriem} manifold $(M,g)$ is \textit{{\rstationary}} if it satisfies the condition 
\begin{equation}
\label{eq_radical_stationary_manifold}
		\kosz(X,Y,\_)\in\annihforms M,
\end{equation}
for any $X,Y\in\fivect{M}$.
\end{definition}

\begin{remark}
The condition from Definition \ref{def_radical_stationary_manifold} means that $\kosz(X,Y,W_p)=0$ for any $X,Y\in \fivect M$ and $W_p\in \fivectnull{M_p}$, $p\in M$.
\end{remark}

\begin{corollary}
\label{thm_Koszul_null_props_rad_stat}
If $(M,g)$ is {\rstationary} and $X,Y\in \fivect M$ and $W\in\fivectnull{M}$, then
\begin{equation}
	\kosz(X,Y,W) = \kosz(Y,X,W) = -\kosz(X,W,Y) = -\kosz(Y,W,X) = 0. \\
\end{equation}
\end{corollary}
\begin{proof}
Follows directly from the Corollary \ref{thm_Koszul_null_props}.
\end{proof}

\begin{remark}
\label{rem_rad_stat_lower_der}
The condition \eqref{eq_radical_stationary_manifold} can be expressed in terms of the lower derivative as
\begin{equation}
		\lderb X Y\in\annihforms M,
\end{equation}
for any $X,Y\in\fivect{M}$.
\end{remark}

\subsection{The covariant derivative of differential 1-forms}
\label{s_cov_der_covect}

For {\nondeg} metrics the covariant derivative of a differential $1$-form is defined in terms of $\der_XY$ \cfeg{GHLF04}{70} by
\begin{equation}
	\left(\der_X\omega\right)(Y) = X\left(\omega(Y)\right) - \omega\left(\der_X Y\right).
\end{equation}
In order to generalize this formula to the case of degenerate metrics, we need to express $\omega\left(\der_XY\right)$ in terms of $\lderb XY$. We can use the identity
\begin{equation}
\label{eq_cov_der_form_raise}
	\omega\left(\der_XY\right) = \metric{\der_XY,\omega^\sharp}
\end{equation}
and rewrite it in a way compatible to the degenerate case as
\begin{equation}
\tag{\ref{eq_cov_der_form_raise}'}
	\omega\left(\der_XY\right) = \metric{\der_XY,\cocontr}\metric{\omega^\sharp,\cocontr}
\end{equation}

\begin{remark}
If the metric is degenerate, we need to be allowed to define the contraction $\kosz(X,Y,\cocontr)\omega(\cocontr)$. This is possible on {\rstationary} singular {\semiriem} manifolds -- since $\lderb XY$ is {\rannih} -- if the differential form $\omega$ is {\rannih} too.
\end{remark}

We can therefore give the following definition:

\begin{definition}
\label{def_cov_der_covect}
Let $(M,g)$ be a {\rstationary} {\semiriem} manifold. We define the covariant derivative of a {\rannih} $1$-form $\omega\in\annihforms{M}$ in the direction of a vector field $X\in\fivect{M}$ by
\begin{equation}
	\der:\fivect{M} \times \annihforms{M} \to \discformsk 1 M
\end{equation}
\begin{equation}
	\left(\der_X\omega\right)(Y) := X\left(\omega(Y)\right) - \annihprod{\lderb X Y,\omega},
\end{equation}
where $\discformsk 1 M$ is the set of sections of $T^*M$ smooth at the points of $M$ where the signature is constant.
\end{definition}

\begin{proposition}
\label{thm_cov_deriv_annih}
If $(M,g)$ is {\rstationary} and $\omega\in\annihforms{M}$ is a {\rannih} $1$-form, then for any $X\in\fivect M$ and $p\in M - \mansigvar M$, $\der_{X_p}\omega_p\in\annih{T_p}M$.
\end{proposition}
\begin{proof}
It follows from the Definition \ref{def_cov_der_covect}. Let $U$ be a neighborhood of $p$ where $g$ has constant signature, and let $W\in\fivectnull U$ so that $W_p\in\radix{T_p}M$. Then, on $U$, $\left(\der_X\omega\right)(W) = X\left(\omega(W)\right) - \annihprod{\lderb X W,\omega} = 0$.
\end{proof}

\begin{corollary}
\label{thm_cov_deriv_annih_smooth}
If $\der_X\omega$ is smooth, then it is a {\rannih} differential $1$-form, $\der_X\omega\in\annihforms M$.
\end{corollary}
\begin{proof}
Follows from Proposition \ref{thm_cov_deriv_annih} because of continuity.
\end{proof}

\begin{definition}
\label{def_cov_der_smooth}
Let $(M,g)$ be a {\rstationary} {\semiriem} manifold. We define
the following vector spaces of differential forms having smooth covariant derivatives:
\begin{equation}
	\srformsk 1 M = \{\omega\in\annihforms M|(\forall X\in\fivect M)\ \der_X\omega\in\annihforms M\},
\end{equation}
\begin{equation}
	\srformsk k M := \bigwedge^k_M\srformsk 1 M.
\end{equation}
\end{definition}

The following theorem extends some properties of the covariant derivative known from the {\nondeg} case \cfeg{ONe83}{59}.

\begin{theorem}
\label{thm_cov_der_covect_props}
The covariant derivative operator $\der$ of differential $1$-forms defined on a {\rstationary} {\semiriem} manifold $(M,g)$ has the following properties:
\begin{enumerate}
	\item \label{thm_cov_der_covect_props_linear}
	It is additive and $\R$-linear in both of its arguments.
	\item \label{thm_cov_der_covect_props_flinearX}
	It is $\fiscal M$-linear in the first argument:

	$\derb{fX}\omega = f\derb X\omega.$
	\item \label{thm_cov_der_covect_props_flinearY}
	It satisfies the \textit{Leibniz rule}:

	$\derb X{f\omega} = f\derb X\omega + X(f) \omega.$
	\item \label{thm_cov_der_covect_props_flat_commut}
	It commutes with the lowering operator:

	$\derb X{Y^\flat} = \lderb XY$.
	\end{enumerate}
	for any $X,Y\in\fivect M$, $\omega\in\annihforms{M}$ and $f\in\fiscal M$.
\end{theorem}
\begin{proof}
The property \eqref{thm_cov_der_covect_props_linear} follows from the direct application of Theorem \ref{thm_l_cov_der_props} to the Definition \ref{def_cov_der_covect}.

For property \eqref{thm_cov_der_covect_props_flinearX},
\begin{equation}
	\derc{fX}\omega Y = fX\left(\omega(Y)\right) - \annihprod{\lderb {fX} Y,\omega} = f \derc X\omega Y.
\end{equation}

Property \eqref{thm_cov_der_covect_props_flinearY} results by
\begin{equation}
\begin{array}{lll}
	\derc{X}{f\omega}Y &=& X\left(f\omega(Y)\right) - \annihprod{\lderb {X} Y,f\omega} \\
	&=& X(f)\omega(Y) +fX\left(\omega(Y)\right) - f\annihprod{\lderb {X} Y,\omega}\\
	&=& f\derc X\omega Y + X(f) \omega(Y).
\end{array}
\end{equation}

For property \eqref{thm_cov_der_covect_props_flat_commut}, we apply Definition \ref{def_cov_der_covect} to $\omega=Y^\flat$. Let $Z\in\fivect{M}$. Then,
\begin{equation}
\begin{array}{lll}
	\derc{X}{Y^\flat}Z &=& X\left(Y^\flat(Z)\right) - \annihprod{\lderb X Z,Y^\flat} \\
	&=& X\metric{Y,Z} - \lderc X Z Y \\
	&=& \lderc X Y Z, \\
\end{array}
\end{equation}
where the last identity follows from Theorem \ref{thm_l_cov_der_props_flinearZ} property \eqref{thm_l_cov_der_props_flinearZ}.
\end{proof}

\begin{corollary}
Let $(M,g)$ be a {\rstationary} {\semiriem} manifold, and
\begin{equation}
	\fiscalannih M=\{f\in\fiscal M|\de f\in\annihformsk 1 M\}.
\end{equation}
Then, $\srformsk k M$ from Definition \ref{def_cov_der_smooth} are $\fiscalannih M$-modules of differential forms.
\end{corollary}
\begin{proof}
From Theorem \ref{thm_cov_der_covect_props} property \eqref{thm_cov_der_covect_props_flinearY} follows that for any $f\in\fiscalannih M$ and $\omega\in\srformsk k M$, $f\omega\in\srformsk k M$.
\end{proof}

\subsection{The covariant derivative of differential forms}
\label{s_cov_der_forms}

We define now the covariant derivative for tensors which are covariant and radical annihilator in all their slots, in particular on differential forms (generalizing the corresponding formulas from the {\nondeg} case, see \eg \citep{GHLF04}{70}).

\begin{definition}
\label{def_cov_der_cov_tensors}
Let $(M,g)$ be a {\rstationary} {\semiriem} manifold. We define the covariant derivative of tensors of type $(0,s)$ as the operator
\begin{equation}
	\der:\fivect{M} \times \otimes^s_M\srformsk 1 M \to \otimes^s_M\annihformsk 1 M
\end{equation}
acting by
\begin{equation}
	\der_X(\omega_1\otimes\ldots\otimes\omega_s) := \der_X(\omega_1)\otimes\ldots\otimes\omega_s +\ldots + \omega_1\otimes\ldots\otimes\der_X(\omega_s)
\end{equation}
\end{definition}

In particular,
\begin{definition}
\label{def_cov_der_forms}
On a {\rstationary} {\semiriem} manifold  $(M,g)$  we define  the covariant derivative of $k$-differential forms by
\begin{equation}
	\der:\fivect{M} \times \srformsk k M \to \annihformsk k M,
\end{equation}
acting by
\begin{equation}
	\der_X(\omega_1\wedge\ldots\wedge\omega_s) := \der_X(\omega_1)\wedge\ldots\wedge\omega_s +\ldots + \omega_1\wedge\ldots\wedge\der_X(\omega_s)
\end{equation}
\end{definition}

\begin{theorem}
\label{thm_cov_der_cov_tensors}
The covariant derivative of a tensor $T\in\otimes^k_M\srformsk 1 M$ on a {\rstationary} {\semiriem} manifold $(M,g)$ satisfies the formula
\begin{equation}
\begin{array}{lll}
	\left(\nabla_X T\right)(Y_1,\ldots,Y_k) &=& X\left(T(Y_1,\ldots,Y_k)\right) \\
	&& - \sum_{i=1}^k\kosz(X,Y_i,\cocontr)T(Y_1,,\ldots,\cocontr,\ldots,Y_k)
\end{array}
\end{equation}
\end{theorem}
\begin{proof}
Because of linearity, it is enough to prove it for the case
\begin{equation}
	T = \omega_1\otimes_M\ldots\otimes_M\omega_k.
\end{equation}
From the Definitions \ref{def_cov_der_cov_tensors} and \ref{def_cov_der_covect},
\begin{equation}
\begin{array}{lll}
	(\der_XT)(Y_1,\ldots,Y_k) &=& 	\der_X(\omega_1\otimes_M\ldots\otimes_M\omega_k)(Y_1,\ldots,Y_k) \\
 &=& \derc X {\omega_1}{Y_1}\cdot\ldots\cdot\omega_k(Y_k) +\ldots \\
 && + \omega_1(Y_1)\cdot\ldots\cdot\derc X {\omega_k}{Y_k} \\
 &=& (X(\omega_1(Y_1)) - \annihprod{\lderb X Y_1,\omega_1})\cdot\ldots\cdot\omega_k(Y_k) +\ldots \\
 && + \omega_1(Y_1)\cdot\ldots\cdot(X(\omega_k(Y_k)) - \annihprod{\lderb X Y_k,\omega_k}) \\
 &=& X(\omega_1(Y_1))\cdot\ldots\cdot\omega_k(Y_k) + \ldots \\
 && +  \omega_1(Y_1)\cdot\ldots\cdot X(\omega_k(Y_k)) \\
 && - \annihprod{\lderb X Y_1,\omega_1}\cdot\ldots\cdot\omega_k(Y_k) \\
 && - \omega_1(Y_1)\cdot\ldots\cdot\annihprod{\lderb X Y_k,\omega_k} \\
 &=&  X\left(T(Y_1,\ldots,Y_k)\right) \\
 && - \ds{\sum_{i=1}^k}\kosz(X,Y_i,\cocontr)T(Y_1,,\ldots,\cocontr,\ldots,Y_k)
\end{array}
\end{equation}
and the desired formula follows.
\end{proof}

\begin{corollary}
\label{thm_cov_der_forms}
Let $(M,g)$ be a {\rstationary} {\semiriem} manifold. The covariant derivative of a $k$-differential form $\omega\in\srformsk kM$ takes the form
\begin{equation}
\begin{array}{lll}
	\left(\nabla_X\omega\right)(Y_1,\ldots,Y_k) &:=& X\left(\omega(Y_1,\ldots,Y_k)\right) \\
	&& - \sum_{i=1}^k\kosz(X,Y_i,\cocontr)\omega(Y_1,,\ldots,\cocontr,\ldots,Y_k)
\end{array}
\end{equation}
\end{corollary}
\begin{proof}
Follows from Theorem \ref{thm_cov_der_cov_tensors}, by verifying that the antisymmetry property of $\omega$ is maintained.
\end{proof}

\begin{corollary}
On a {\rstationary} {\semiriem} manifold $(M,g)$, the metric $g$ is parallel:
\begin{equation}
	\nabla_Xg = 0.
\end{equation}
\end{corollary}
\begin{proof}
Follows from Theorems \ref{thm_cov_der_cov_tensors} and \ref{thm_Koszul_form_props}, property \eqref{thm_Koszul_form_props_commutYZ}:
\begin{equation}
	(\nabla_Xg)(Y,Z) = X\metric{Y,Z} - \kosz(X,Y,\cocontr)g(\cocontr,Z) - \kosz(X,Z,\cocontr)g(Y,\cocontr) = 0.
\end{equation}
\end{proof}

\subsection{{\ssemireg} {\semiriem} manifolds}
\label{s_semi_regular}

An important particular type of {\rstationary} {\semiriem} manifold is provided by the {\semireg} {\semiriem} manifolds, introduced below.

\begin{definition}
\label{def_semi_regular_semi_riemannian}
A \textit{{\semireg} {\semiriem} manifold} is a singular {\semiriem} manifold $(M,g)$ which satisfies
\begin{equation}
	\ldera X Y \in\srformsk 1 M
\end{equation}
for any vector fields $X,Y\in\fivect{M}$.
\end{definition}

\begin{remark}
\label{rem_semi_regular_semi_riemannian}
By Definition \ref{def_cov_der_smooth}, this is equivalent to saying that for any $X,Y,Z\in\fivect M$
\begin{equation}
	\dera X {\ldera Y}Z \in \annihforms M.
\end{equation}
\end{remark}

\begin{remark}
Recall that $\srformsk 1 M \subseteq \annihforms M$. This means that any {\semireg} {\semiriem} manifold is also {\rstationary} (\cf Definition \ref{def_radical_stationary_manifold}).
\end{remark}

\begin{proposition}
\label{thm_sr_cocontr_kosz}
Let $(M,g)$ be a {\rstationary} {\semiriem} manifold. Then, the manifold $(M,g)$ is {\semireg} if and only if for any $X,Y,Z,T\in\fivect M$
\begin{equation}
	\kosz(X,Y,\cocontr)\kosz(Z,T,\cocontr) \in \fiscal M.
\end{equation}
\end{proposition}
\begin{proof}
From the Definition \ref{def_cov_der_covect} of the covariant derivative of $1$-forms we obtain that
\begin{equation}
\begin{array}{lll}
	\derc X {\lderb YZ} T 
	&=& X\left(\lderc Y Z T\right) - \annihprod{\lderb X T,\lderb YZ} \\
	&=& X\left(\lderc Y Z T\right) - \kosz(X,T,\cocontr)\kosz(Y,Z,\cocontr). \\
\end{array}
\end{equation}
It follows that $\derc X {\lderb YZ} T$ is smooth if and only if $\kosz(X,T,\cocontr)\kosz(Y,Z,\cocontr)$ is.
\end{proof}

\section{Curvature of {\semireg} {\semiriem} manifolds}
\label{s_riemann_curvature}

The standard way to define the curvature invariants is to construct the Levi-Civita connection of the metric \cfeg{ONe83}{59}, and from this the curvature operator \cfeg{ONe83}{74}. The Ricci tensor and the scalar curvature \cfeg{ONe83}{87--88} follow by contraction \cfeg{ONe83}{83}.

Unfortunately, in the case of singular {\semiriem} manifolds the usual road is not available, because there is no intrinsic Levi-Civita connection. But, as we shall see in this section, the Riemann curvature tensor can be obtained from the lower covariant derivative and the covariant derivative of {\rannih} differential forms. 
For {\rstationary} manifolds the Riemann curvature tensor thus introduced is guaranteed to be smooth only on the regions of constant signature, but for {\semireg} manifolds it is smooth everywhere.

In order to obtain the Ricci curvature tensor, and further the scalar curvature, we need to contract the Riemann curvature tensor in two covariant indices. Because the metric may be degenerate, this covariant contraction can be defined only if the Riemann curvature tensor is {\rannih} in its slots. We will see that this is the case, and in \sref{s_ricci_tensor_scalar} we define the Ricci tensor and the scalar curvature.

\subsection{Riemann curvature of {\semireg} {\semiriem} manifolds}
\label{ss_riemann_curvature}

\begin{definition}
\label{def_riemann_curvature_operator}
Let $(M,g)$ be a {\rstationary} {\semiriem} manifold. We define the \textit{lower Riemann curvature operator} as
\begin{equation}
	\curv{}{}: \fivect M ^3 \to \discformsk 1 M
\end{equation}
\begin{equation}
\label{eq_riemann_curvature_operator}
	\curv XY Z := \dera X {\ldera Y}Z - \dera Y {\ldera X}Z - \ldera {[X,Y]}Z
\end{equation}
for any vector fields $X,Y,Z\in\fivect{M}$.
\end{definition}

\begin{definition}
\label{def_riemann_curvature}
We define the \textit{Riemann curvature tensor} as
\begin{equation}
	R: \fivect M\times \fivect M\times \fivect M\times \fivect M \to \R,
\end{equation}
\begin{equation}
\label{eq_riemann_curvature}
	R(X,Y,Z,T) := (\curv XY Z)(T)
\end{equation}
for any vector fields $X,Y,Z,T\in\fivect{M}$.
\end{definition}

\begin{remark}
The Riemann curvature tensor from Definition \ref{def_riemann_curvature} generalizes the Riemann curvature tensor $R(X,Y,Z,T) := \metric{R_{XY}Z,T}$ known from {\semiriem} geometry \cfeg{ONe83}{75}.
\end{remark}

\begin{remark}
It follows from the Definition \ref{def_riemann_curvature} that
\begin{equation}
\label{eq_riemann_curvature_explicit}
	R(X,Y,Z,T) = \derc X {{\ldera Y}Z}T - \derc Y {{\ldera X}Z}T - \lderc {[X,Y]}ZT
\end{equation}
for any vector fields $X,Y,Z,T\in\fivect{M}$.
\end{remark}

\begin{theorem}
\label{thm_riemann_curvature_semi_regular}
Let $(M,g)$ be a {\semireg} {\semiriem} manifold. The Riemann curvature is a smooth tensor field $R\in\tensors 0 4 M$.
\end{theorem}
\begin{proof}
Remember from Theorem \ref{thm_l_cov_der_props}, property \eqref{thm_l_cov_der_props_linear} that the lower covariant derivative for vector fields is additive and $\R$-linear in both of is arguments. From the same Theorem \ref{thm_cov_der_covect_props} property \eqref{thm_l_cov_der_props_linear}, we recall that the covariant derivative for differential $1$-forms is additive and $\R$-linear in both of is arguments. By combining the two, it follows \textit{the additivity and $\R$-linearity} of the Riemann curvature $R$ in all of its four arguments.

We will show now that $R$ is $\fiscal M$-linear in its four arguments. The proof goes almost similar to the {\nondeg} case, but we will give it explicitly, because in our proof we need to avoid any use of the Levi-Civita connection or of the inverse of the metric tensor, for example index raising.

We apply the properties of the lower covariant derivative for vector fields, as exposed in Theorem \ref{thm_cov_der_covect_props} properties \eqref{thm_Koszul_form_props_flinearX}-\eqref{thm_Koszul_form_props_flinearZ}, and those of the covariant derivative for differential $1$-forms, as known from Theorem \ref{thm_cov_der_covect_props}, properties \eqref{thm_cov_der_covect_props_flinearX}-\eqref{thm_cov_der_covect_props_flat_commut}, to verify that for any function $f\in\fiscal{M}$, $R(fX,Y,Z,T)=R(X,fY,Z,T)=R(X,Y,fZ,T)=R(X,Y,Z,fT)=fR(X,Y,Z,T)$.

Since $[fX,Y]=f[X,Y]-Y(f)X$,
\begin{equation*}
\begin{array}{lll}
	R(fX,Y,Z,T) &=& \derc {fX} {{\ldera Y}Z}T - \derc Y {{\ldera {fX}}Z}T - \lderc {[fX,Y]}ZT \\
	&=& f\derc {X} {{\ldera Y}Z}T - \derc Y {{(f\ldera {X}}Z)}T \\
	&& - \lderc {f[X,Y]-Y(f)X}ZT \\
	&=& f\derc {X} {{\ldera Y}Z}T - f\derc Y {{\ldera {X}}Z}T \\
	&& - Y(f)\lderc X Z T - f\lderc {[X,Y]}ZT \\
	&&  + Y(f)\lderc {X}ZT \\
	&=& fR(X,Y,Z,T). \\
\end{array}
\end{equation*}
The Definition \ref{def_riemann_curvature} implies that $R(X,Y,Z,T)=-R(Y,X,Z,T)$, which leads immediately to 
\begin{equation}
	R(X,fY,Z,T)=fR(X,Y,Z,T).
\end{equation}
\begin{equation*}
\begin{array}{lll}
	R(X,Y,fZ,T) &=& \derc {X} {{\ldera Y}{fZ}}T - \derc Y {{\ldera {X}}{fZ}}T - \lderc {[X,Y]}{fZ}T \\
	 &=& \derc {X} {(f{\ldera Y}{Z}+Y(f)Z)}T \\
	 && - \derc {Y} {(f{\ldera X}{Z}+X(f)Z)}T\\
	 && - (f{\ldera {[X,Y]}}{Z}+[X,Y](f)Z^\flat)(T)\\
	 &=& \derc {X} {(f{\ldera Y}{Z})}T + \derc {X} {(Y(f)Z^\flat)}T \\
	 && - \derc {Y} {(f{\ldera X}{Z})}T - \derc {Y} {(X(f)Z^\flat)}T \\
	 && - f({\ldera {[X,Y]}}{Z})(T)-[X,Y](f)Z^\flat(T)\\
	 &=& f\derc {X} {{\ldera Y}{Z}}T + X(f) {({\ldera Y}{Z})}(T) \\
	 && + X(Y(f)) {(Z^\flat)}(T)  + Y(f)\derc {X} {Z^\flat}T \\
	 && - f\derc {Y} {{\ldera X}{Z}}T - Y(f) {({\ldera X}{Z})}(T) \\
	 && - Y(X(f)) {(Z^\flat)}(T)  - X(f)\derc {Y} {Z^\flat}T \\
	 && - f({\ldera {[X,Y]}}{Z})(T)-[X,Y](f)Z^\flat(T)\\
	&=& fR(X,Y,Z,T). \\
\end{array}
\end{equation*}
The $\fiscal M$-linearity in $T$ follows from the definition of $R$, observing that $\derb {X} {{\ldera Y}{Z}}$, $\derb Y {{\ldera {X}}{Z}}$ and $\lderb {[X,Y]}{Z}$ are in fact differential $1$-forms.

The lower covariant derivative of a smooth vector field is a smooth differential $1$-form on $M$, therefore ${\ldera X}{Z}$, ${\ldera Y}{Z}$ and ${\ldera {[X,Y]}}{Z}$ are smooth on $M$. It follows that $R$ is also smooth on $M$.
\end{proof}

\begin{remark}
One can write
\begin{equation}
	\curv{}{}: \fivect M ^2 \to \tensors 0 2 M
\end{equation}
\begin{equation}
	\curv XY := \dera X {\ldera Y} - \dera Y {\ldera X} - \ldera {[X,Y]},
\end{equation}
with the amendment that
\begin{equation}
	\curv XY(Z,T) := (\curv XYZ)(T)
\end{equation}
for any $Z,T\in\fivect M$.
\end{remark}

\subsection{The symmetries of the Riemann curvature tensor}
\label{s_riemann_curvature_symmetries}

The following proposition generalizes well-known symmetry properties of the Riemann curvature tensor of a {\nondeg} metric \cfeg{ONe83}{75} to  {\semireg} metrics. The proofs are similar to the {\nondeg} case, except that they avoid using the covariant derivative and the index raising, so we prefer to give them explicitly.

\begin{proposition}[The symmetries of the Riemann curvature]
\label{thm_curv_symm}
Let $(M,g)$ be a {\semireg} {\semiriem} manifold. Then, for any $X,Y,Z,T\in\fivect M$, the Riemann curvature has the following symmetry properties
\begin{enumerate}
	\item \label{thm_curv_symm_xy}
	$\curv XY = -\curv YX$
	\item \label{thm_curv_symm_zt}
	$\curv XY(Z,T) = -\curv XY(T,Z)$
	\item  \label{thm_curv_symm_xyz}
	$\curv YZ X + \curv ZX Y + \curv XY Z = 0$
	\item  \label{thm_curv_symm_xy_zt}
	$\curv XY(Z,T) = \curv ZT(X,Y)$
\end{enumerate}
\end{proposition}
\begin{proof}
\eqref{thm_curv_symm_xy} Follows from the Definition \ref{def_riemann_curvature_operator}:
\begin{equation*}
\begin{array}{lll}
	\curv XY Z&=& \dera X {\ldera Y}Z - \dera Y {\ldera X}Z - \ldera {[X,Y]}Z \\
	&=&-\curv YX Z
\end{array}
\end{equation*}

\eqref{thm_curv_symm_zt}
This is equivalent to
\begin{equation}
	\curv XY(V,V)=0
\end{equation}
for any $V\in\fivect M$.
From the property of the lower covariant derivative of being metric (Theorem \ref{thm_l_cov_der_props}, property \eqref{thm_l_cov_der_props_flinearZ}) it follows that
\begin{equation*}
\lderc{[X,Y]}VV=\frac 1 2[X,Y]\metric{V,V}
\end{equation*}
and
\begin{equation*}
X(\lderc YVV) = \dsfrac 1 2 XY\metric{V,V}.
\end{equation*}
From the Definition \ref{def_cov_der_covect} of the covariant derivative of $1$-forms we obtain that
\begin{equation}
	\derc X {\lderb YV} V  = X\left(\lderc Y V V\right) - \annihprod{\lderb X V,\lderb YV}.
\end{equation}
By combining them we get
\begin{equation}
	\derc X {\lderb YV} V  = \dsfrac 1 2 XY\metric{V,V} - \annihprod{\lderb X V,\lderb YV}.
\end{equation}
Therefore,
\begin{equation*}
\begin{array}{lll}
	\curv XY(V,V) 
	&=& \derc X {{\ldera Y}V}V - \derc Y {{\ldera X}V}V - \lderc {[X,Y]}VV \\
	&=& \dsfrac 1 2X\left(\lderc Y V V\right) - \annihprod{\lderb X V,\lderb YV} \\
	&& - \dsfrac 1 2Y\left(\lderc X V V\right) + \annihprod{\lderb Y V,\lderb XV} \\
	&& - \frac 1 2[X,Y]\metric{V,V} = 0\\
\end{array}
\end{equation*}

\eqref{thm_curv_symm_xyz}
As the proof of this identity usually goes, we define the cyclic sum for any $F:\fivect M^3\to\fiformk 1 M$ by
\begin{equation}
\begin{array}{l}
	\sum_{\cyclic}F(X,Y,Z):=F(X,Y,Z)+F(Y,Z,X)+F(Z,X,Y)
\end{array}
\end{equation}
and observe that it doesn't change at cyclic permutations of $X,Y,Z$. Then, from the properties of the lower covariant derivative and from Jacobi's identity,
\begin{equation*}
\begin{array}{lll}
\sum_{\cyclic}\curv XY Z &=& \sum_{\cyclic}\dera X {\ldera Y}Z - \sum_{\cyclic}\dera Y {\ldera X}Z - \sum_{\cyclic}\ldera {[X,Y]}Z\\
&=& \sum_{\cyclic}\dera X {\ldera Y}Z - \sum_{\cyclic}\dera X {\ldera Z}Y - \sum_{\cyclic}\ldera {[X,Y]}Z\\
&=& \sum_{\cyclic}\dera X \left({\ldera Y}Z - {\ldera Z}Y\right) - \sum_{\cyclic}\ldera {[X,Y]}Z\\
&=& \sum_{\cyclic}\dera X [Y,Z]^\flat - \sum_{\cyclic}\ldera {[X,Y]}Z\\
&=& \sum_{\cyclic}\dera X^\flat [Y,Z] - \sum_{\cyclic}\ldera {[Y,Z]}X\\
&=& \sum_{\cyclic}[X,[Y,Z]]^\flat = 0.\\
\end{array}
\end{equation*}

To show \eqref{thm_curv_symm_xy_zt} we apply  \eqref{thm_curv_symm_xyz} four times (as in the usual proof of the properties of the curvature):
\begin{equation*}
\begin{array}{lllllll}
	\curv XY(Z,T) &+& \curv YZ(X,T) &+& \curv ZX(Y,T) &=& 0 \\
	\curv YZ(T,X) &+& \curv ZT(Y,X) &+& \curv TY(Z,X) &=& 0 \\
	\curv ZT(X,Y) &+& \curv TX(Z,Y) &+& \curv XZ(T,Y) &=& 0 \\
	\curv TX(Y,Z) &+& \curv XY(T,Z) &+& \curv YT(X,Z) &=& 0 \\
\end{array}
\end{equation*}
then sum up, divide by $2$ and get:
\begin{equation*}
\curv XY(Z,T) = \curv ZT(X,Y).
\end{equation*}
\end{proof}

\begin{corollary}[see \citep{Kup87b}{270}]
\label{thm_curvature_tensor_radical}
For any $X,Y,Z\in \fivect{M}$ and $W\in\fivectnull{M}$, the Riemann curvature tensor $R$ satisfies
\begin{equation}
	R(W,X,Y,Z) = R(X,W,Y,Z) = R(X,Y,W,Z) = R(X,Y,Z,W) = 0.
\end{equation}
\end{corollary}
\begin{proof}
From the Remark \ref{rem_semi_regular_semi_riemannian}, $\dera X {\ldera Y}Z \in \annihforms M$, and from the Remark \ref{rem_rad_stat_lower_der}, $\lderb X Y\in\annihforms M$, for any $X,Y,Z\in\fivect{M}$. Therefore, $R(X,Y,Z,W)=0$. From the symmetry properties \eqref{thm_curv_symm_xy} and \eqref{thm_curv_symm_xy_zt} from Theorem \ref{thm_curv_symm}, this property extends to all other slots of the Riemann curvature tensor.
\end{proof}

\begin{corollary}
\label{thm_curv_annih}
Let $(M,g)$ be a {\semireg} {\semiriem} manifold. Then, for any $X,Y\in\fivect M$,
	$	\curv XY\in\annihformsk 2 M$ ($\curv XY$ is a {\rannih}). 
\end{corollary}
\begin{proof}
Follows from the Corollary \ref{thm_curvature_tensor_radical}.
\end{proof}

\subsection{Ricci curvature tensor and scalar curvature}
\label{s_ricci_tensor_scalar}

In {\nondeg} {\semiriem} geometry, the Ricci tensor is obtained by tracing the Riemann curvature, and the scalar curvature by tracing the Ricci tensor \cfeg{ONe83}{87--88}. In the degenerate case, an invariant contraction can be performed only on {\rannih} slots. Fortunately, this is the case of the Riemann tensor even in the case when the metric is degenerate (Corollary \ref{thm_curvature_tensor_radical}), so it is possible to define the Ricci tensor as:

\begin{definition}
\label{def_ricci_curvature_tensor}
Let $(M,g)$ be a {\rstationary} singular {\semiriem} manifold with constant signature. The \textit{Ricci curvature tensor} is defined as the covariant contraction of the Riemann curvature tensor
\begin{equation}
	\ric(X,Y):=R(X,\cocontr,Y,\cocontr)
\end{equation}
for any $X,Y\in\fivect{M}$.
\end{definition}

The symmetry of the Ricci tensor works just like in the {\nondeg} case \cfeg{ONe83}{87}:

\begin{proposition}
The Ricci curvature tensor on a {\rstationary} singular {\semiriem} manifold with constant signature is symmetric:
\begin{equation}
	\ric(X,Y)=\ric(Y,X)
\end{equation}
for any $X,Y\in\fivect{M}$.
\end{proposition}
\begin{proof}
The Proposition \ref{thm_curv_symm} states that $R(X,Y,Z,T)=R(Z,T,X,Y)$ for any $X,Y,Z,T\in\fivect{M}$. Therefore, $\ric(X,Y)=\ric(Y,X)$.
\end{proof}

The scalar curvature is obtained from the Ricci tensor like in the {\nondeg} case \cfeg{ONe83}{88}:

\begin{definition}
\label{def_scalar_curvature}
Let $(M,g)$ be a {\rstationary} singular {\semiriem} manifold with constant signature. The \textit{scalar curvature} is defined as the covariant contraction of the Ricci curvature tensor
\begin{equation}
	s:=\ric(\cocontr,\cocontr).
\end{equation}
\end{definition}

\begin{remark}
The Ricci and the scalar curvatures are smooth for the case of {\rstationary} singular {\semiriem} manifolds having the metric with constant signature. For {\semireg} {\semiriem} manifolds, the Ricci and scalar curvatures are smooth in the regions of constant curvature, and become in general divergent as we approach the points where the signature changes.
\end{remark}

\section{Curvature of {\semireg} {\semiriem} manifolds II}
\label{s_riemann_curvature_ii}

This section contains some complements on the Riemann curvature tensor of {\semireg} {\semiriem} manifolds. A useful formula of this curvature in terms of the Koszul form is provided in \sref{s_riemann_curvature_koszul_formula}.

In the subsection \sref{s_koszul_deriv_curv_funct} we recall some results from \cite{Kup87b} concerning the (non-unique) Koszul derivative $\der$ and the associated curvature function $R_\der$, and show that $\metric{R_\der(\_,\_)\_,\_}$ coincides with that of the Riemann curvature tensor given in this article in \sref{s_riemann_curvature}.

\subsection{Riemann curvature in terms of the Koszul form}
\label{s_riemann_curvature_koszul_formula}

\begin{proposition}
\label{thm_riemann_curvature_tensor_koszul_formula}
For any vector fields $X,Y,Z,T\in\fivect{M}$ on a {\semireg} {\semiriem} manifold $(M,g)$:
\begin{equation}
\begin{array}{lll}
	R(X,Y,Z,T) &=& X\left(\lderc Y Z T\right) - Y\left(\lderc X Z T\right) - \lderc {[X,Y]}ZT \\
&& + \annihprod{\lderb XZ,\lderb Y T} - \annihprod{\lderb YZ,\lderb X T} \\
\end{array}
\end{equation}
and, alternatively,
\begin{equation}
\label{eq_riemann_curvature_tensor_koszul_formula}
\begin{array}{lll}
	R(X,Y,Z,T)&=& X \kosz(Y,Z,T) - Y \kosz(X,Z,T) - \kosz([X,Y],Z,T)\\
	&& + \kosz(X,Z,\cocontr)\kosz(Y,T,\cocontr) - \kosz(Y,Z,\cocontr)\kosz(X,T,\cocontr)
\end{array}
\end{equation}
\end{proposition}
\begin{proof}
From the Definition \ref{def_cov_der_covect} of the covariant derivative of $1$-forms we obtain that
\begin{equation}
	\derc X {\lderb YZ} T  = X\left(\lderc Y Z T\right) - \annihprod{\lderb X T,\lderb YZ},
\end{equation}
therefore
\begin{equation}
\begin{array}{lll}
	R(X,Y,Z,T) &=& \derc X {{\ldera Y}Z}T - \derc Y {{\ldera X}Z}T -
 \lderc {[X,Y]}ZT \\
&=& X\left(\lderc Y Z T\right) - Y\left(\lderc X Z T\right) - \lderc {[X,Y]}ZT \\
&& + \annihprod{\lderb XZ,\lderb Y T} - \annihprod{\lderb YZ,\lderb X T} \\
\end{array}
\end{equation}
for any vector fields $X,Y,Z,T\in\fivect{M}$.
The second formula \eqref{eq_riemann_curvature_tensor_koszul_formula} follows from the definition of the lower derivative of vector fields.
\end{proof}

\begin{remark}
In a coordinate basis, the components of the Riemann curvature tensor are given by
\begin{equation}
\label{eq_riemann_curvature_tensor_coord}
	R_{abcd}= \partial_a \kosz_{bcd} - \partial_b \kosz_{acd} + \annihg^{st}(\kosz_{acs}\kosz_{bdt} - \kosz_{bcs}\kosz_{adt}).
\end{equation}\end{remark}
\begin{proof}
\begin{equation}
\begin{array}{lll}
	R_{abcd}&:=& R(\partial_a,\partial_b,\partial_c,\partial_d)\\
	&=& \partial_a \kosz(\partial_b,\partial_c,\partial_d) - \partial_b \kosz(\partial_a,\partial_c,\partial_d) - \kosz([\partial_a,\partial_b],\partial_c,\partial_d)\\
	&& + \kosz(\partial_a,\partial_c,\cocontr)\kosz(\partial_b,\partial_d,\cocontr) - \kosz(\partial_b,\partial_c,\cocontr)\kosz(\partial_a,\partial_d,\cocontr)\\
	&=& \partial_a \kosz_{bcd} - \partial_b \kosz_{acd} + \annihg^{st}(\kosz_{acs}\kosz_{bdt} - \kosz_{bcs}\kosz_{adt})
\end{array}
\end{equation}
\end{proof}

\subsection{Relation with Kupeli's curvature function}
\label{s_koszul_deriv_curv_funct}

Through the work of Demir Kupeli \cite{Kup87b} we have seen that for a {\rstationary} singular {\semiriem} manifold (with constant signature) $(M,g)$ there is always a Koszul derivative $\der$, from whose curvature function $R_\der$ we can construct a tensor field $\metric{R_\der(\_,\_)\_,\_}$. We may wonder how is $\metric{R_\der(\_,\_)\_,\_}$ related to the Riemann curvature tensor from the Definition \ref{def_riemann_curvature}. We will see that they coincide for a {\rstationary} singular {\semiriem} manifold.

\begin{definition}[Koszul derivative, \cf Kupeli \citep{Kup87b}{261}]
\label{def_Koszul_derivative}
A \textit{Koszul derivative} on a {\rstationary} {\semiriem} manifold with constant signature is an operator $\der:\fivect M\times\fivect M\to \fivect M$ which satisfies the \textit{Koszul formula}
\begin{equation}
\label{eq_Koszul_formula}
\begin{array}{llll}
	\metric{\der_X Y,Z} &=& \kosz(X,Y,Z).
\end{array}
\end{equation}
\end{definition}

\begin{remark}[\cf Kupeli \citep{Kup87b}{262}]
The Koszul derivative corresponds, for the {\nondeg} case, to the Levi-Civita connection.
\end{remark}

\begin{definition}[Curvature function, \cf Kupeli \citep{Kup87b}{266}]
\label{def_curvature_function}
The \textit{curvature function} $R_\der : \fivect M\times\fivect M\times\fivect M\to \fivect M$ of a Koszul derivative $\der$ on a singular {\semiriem} manifold with constant signature $(M,g)$ is defined by
\begin{equation}
\label{eq_curvature_function}
R_\der(X,Y)Z:=\der_X\der_Y Z - \der_Y \der_X Z - \der_{[X,Y]}Z.
\end{equation}
\end{definition}

\begin{remark}
In \citep{Kup87b}{266-268} it is shown that $\metric{R_\der(\_,\_)\_,\_}\in\tensors 0 4 {M}$ and it has the same symmetry properties as the Riemann curvature tensor of a Levi-Civita connection.
\end{remark}

\begin{theorem}
Let $(M,g)$ be a {\rstationary} singular {\semiriem} manifold with constant signature, and $\der$ a Koszul derivative on $M$. The Riemann curvature tensor is related to the curvature function by
\begin{equation}
	\metric{R_\der(X,Y)Z,T} = R(X,Y,Z,T)
\end{equation}
for any $X,Y,Z,T\in\fivect{M}$.
\end{theorem}
\begin{proof}
From Theorem \ref{thm_Koszul_form_props} and Definition \ref{def_curvature_function}, applying the property of contraction with the metric from Lemma \ref{thm_contraction_with_metric} and the Koszul formula for the Riemann curvature tensor \eqref{eq_riemann_curvature_tensor_koszul_formula}, we obtain
\begin{equation*}
\begin{array}{lll}
\metric{R_\der(X,Y)Z,T} &=&\metric{\der_X \der_Y Z,T} - \metric{\der_Y \der_X Z,T} - \metric{\der_{[X,Y]}Z,T}\\
&=&X \metric{\der_YZ,T} - \metric{\der_Y Z, \der_X T} \\
	&& - Y \metric{\der_X Z,T} + \metric{\der_X Z, \der_Y T} - \metric{\der_{[X,Y]} Z,T} \\
	&=& X \kosz(Y,Z,T) - \kosz(Y,Z,\cocontr)\kosz(X,T,\cocontr) \\
	&& - Y \kosz(X,Z,T) + \kosz(X,Z,\cocontr)\kosz(Y,T,\cocontr)\\
	&& - \kosz([X,Y],Z,T) \\
&=& R(X,Y,Z,T)
\end{array}
\end{equation*}
\end{proof}

\section{Examples of {\semireg} {\semiriem} manifolds}
\label{s_semi_reg_semi_riem_man_example}

\subsection{Diagonal metric}
\label{s_semi_reg_semi_riem_man_example_diagonal}

Let $(M,g)$ be a singular {\semiriem} manifold with variable signature having the property that for each point $p\in M$ there is a local coordinate system around $p$ in which the metric takes a diagonal form $g=\diag{(g_{11},\ldots,g_{nn})}$. According to equation \eqref{eq_Koszul_form_coord}, $2\kosz_{abc}=\partial_a g_{bc} + \partial_b g_{ca} - \partial_c g_{ab}$, but since $g$ is diagonal, we have only the following possibilities: $\kosz_{baa} = \kosz_{aba} = -\kosz_{aab} = \frac 1 2\partial_b g_{aa}$, for $a\neq b$, and $\kosz_{aaa} = \frac 1 2\partial_a g_{aa}$.

The manifold $(M,g)$ is {\rstationary} if and only if whenever $g_{aa}=0$, $\partial_b g_{aa} = \partial_a g_{bb} = 0$.

According to Proposition $\ref{thm_sr_cocontr_kosz}$, the manifold $(M,g)$ is {\semireg} if and only if
\begin{equation}
\label{eq_diag_g_contr_kosz}
	\sum_{\substack{s\in\{1,\ldots,n\} \\ g_{ss}\neq 0}} \dsfrac{\partial_a g_{ss}\partial_b g_{ss}}{g_{ss}},
	\sum_{\substack{s\in\{1,\ldots,n\} \\ g_{ss}\neq 0}} \dsfrac{\partial_s g_{aa}\partial_s g_{bb}}{g_{ss}},
	\sum_{\substack{s\in\{1,\ldots,n\} \\ g_{ss}\neq 0}} \dsfrac{\partial_a g_{ss}\partial_s g_{bb}}{g_{ss}}
\end{equation}
are all smooth.

One way to ensure this is for instance if the functions $u,v:M\to\R$ defined as
\begin{equation}
	u(p):=\Bigg\{
\begin{array}{ll}
\dsfrac{\partial_b g_{aa}}{\sqrt{\abs{g_{aa}}}} & g_{aa}\neq 0 \\
0 & g_{aa}= 0 \\
\end{array}
\tn{ and }
	v(p):=\Bigg\{
\begin{array}{ll}
\dsfrac{\partial_a g_{bb}}{\sqrt{\abs{g_{aa}}}} & g_{aa}\neq 0 \\
0 & g_{aa}= 0 \\
\end{array}
\end{equation}
and $\sqrt{\abs{g_{aa}}}$ are smooth for all $a,b\in\{1,\ldots,n\}$.
In this case it is easy to see that all the terms of the sums in equation \eqref{eq_diag_g_contr_kosz} are smooth.

\subsection{Conformally-{\nondeg} metrics}
\label{s_semi_reg_semi_riem_man_example_conformal}

Another example of {\semireg} metric is given by those that can be obtained by a conformal transformation \cfeg{HE95}{42} from {\nondeg} metrics.

\begin{definition}
A singular {\semiriem} manifold $(M,g)$ is said to be \textit{conformally {\nondeg}} if there is a {\nondeg} {\semiriem} metric $\tilde g$ on $M$ and a smooth function $\Omega\in\fiscal M$, $\Omega\geq 0$, so that $g(X,Y)=\Omega^2\tilde g(X,Y)$ for any $X,Y\in\fivect M$. The manifold $(M,g)$ is alternatively denoted by $(M,\tilde g, \Omega)$. A singularity of a conformally {\nondeg} manifold is called \textit{isotropic singularity}.
\end{definition}

The following proposition shows what happens to the Koszul form at a conformal transformation of the metric, similar to the {\nondeg} case \cfeg{HE95}{42}.

\begin{proposition}
\label{thm_conformal_koszul_form}
Let $(M,\tilde g, \Omega)$ be a conformally {\nondeg} singular {\semiriem} manifold. Then, the Koszul form $\kosz$ of $g$ is related to the Koszul form $\tilde \kosz$ of $\tilde g$ by:
\begin{equation}
\label{eq_conformal_koszul_form}
\kosz(X,Y,Z) = \Omega^2\tilde \kosz(X,Y,Z) + \Omega\left[\tilde g(Y,Z)X + \tilde g(X,Z)Y - \tilde g(X,Y)Z\right](\Omega)
\end{equation}
\end{proposition}
\begin{proof}
From the Koszul formula we obtain
\begin{equation*}
\begin{array}{llll}
	\kosz(X,Y,Z) &=&\ds{\frac 1 2} \{ X (\Omega^2\tilde g(Y,Z)) + Y (\Omega^2\tilde g(Z,X)) - Z (\Omega^2\tilde g(X,Y)) \\
	&&\ - \Omega^2\tilde g(X,[Y,Z]) + \Omega^2\tilde g(Y, [Z,X]) + \Omega^2\tilde g(Z, [X,Y])\} \\
	&=&\ds{\frac 1 2} \{ \Omega^2X (\tilde g(Y,Z)) + \tilde g(Y,Z)X(\Omega^2) + \Omega^2Y (\tilde g(X,Z)) \\
	&& + \tilde g(X,Z)Y(\Omega^2) - \Omega^2Z (\tilde g(X,Y)) - \tilde g(X,Y)Z(\Omega^2) \\
	&&\ - \Omega^2\tilde g(X,[Y,Z]) + \Omega^2\tilde g(Y, [Z,X]) + \Omega^2\tilde g(Z, [X,Y])\} \\
	&=& \Omega^2 \tilde \kosz(X,Y,Z) + \ds{\frac 1 2} \{ \tilde g(Y,Z)X(\Omega^2) \\
	&& + \tilde g(X,Z)Y(\Omega^2) - \tilde g(X,Y)Z(\Omega^2)\} \\
	&=& \Omega^2\tilde \kosz(X,Y,Z) + \Omega\big[\tilde g(Y,Z)X \\
	&& + \tilde g(X,Z)Y - \tilde g(X,Y)Z\big](\Omega)
\end{array}
\end{equation*}
\end{proof}

\begin{theorem}
\label{thm_conformal_semi_regular}
Let $(M,\tilde g, \Omega)$ be a singular {\semiriem} manifold which is conformally {\nondeg}. Then, $(M,g=\Omega^2\tilde g)$ is a {\semireg} {\semiriem} manifold.
\end{theorem}
\begin{proof}
The metric $g$ is either {\nondeg}, or it is $0$. Therefore, the manifold $(M,g)$ is {\rstationary}.

Let $(E_a)_{a=1}^n$ be a local frame of vector fields on an open $U\subseteq M$, which is orthonormal with respect to the {\nondeg} metric $\tilde g$. Then, the metric $g$ is diagonal in $(E_a)_{a=1}^n$. 

Proposition \ref{thm_conformal_koszul_form} implies that the Koszul form has the form $\kosz(X,Y,Z) = \Omega h(X,Y,Z)$, where
\begin{equation}
	h(X,Y,Z) = \Omega \tilde \kosz(X,Y,Z) + \left[\tilde g(Y,Z)X + \tilde g(X,Z)Y - \tilde g(X,Y)Z\right](\Omega)
\end{equation}
is a smooth function depending on $X,Y,Z$. Moreover, if $\Omega=0$, then $h(X,Y,Z)=0$ as well, because the first term is multiple of $\Omega$, and the second is a partial derivative of $\Omega$, which reaches its minimum at $0$.

Theorem \ref{thm_contraction_orthogonal} saids that, on the regions of constant signature, if $r=n-\rank g+1$, for any $X,Y,Z,T\in U$ and for any $a\in\{1,\ldots,n\}$, 
\begin{equation}
\begin{array}{lll}
\kosz(X,Y,\cocontr)\kosz(Z,T,\cocontr)
&=& \sum_{a=r}^n \dsfrac{\kosz(X,Y,E_a)\kosz(Z,T,E_a)}{g(E_a,E_a)} \\
&=& \sum_{a=r}^n \dsfrac{\Omega^2 h(X,Y,E_a) h(Z,T,E_a)}{\Omega^2\tilde g(E_a,E_a)} \\
&=& \sum_{a=1}^n \dsfrac{h(X,Y,E_a) h(Z,T,E_a)}{\tilde g(E_a,E_a)}. \\
\end{array}
\end{equation}
If $\Omega=0$, then $h(X,Y,Z)=0$, therefore the last member does not depend on $r$. It follows that $\kosz(X,Y,\cocontr)\kosz(Z,T,\cocontr)\in\fiscal M$, and according to Proposition \ref{thm_sr_cocontr_kosz}, $(M,g)$ is {\semireg}.
\end{proof}

\section{Einstein's equation on {\semireg} spacetimes}
\label{s_einstein_tensor_densitized}

\subsection{The problem of singularities}
\label{s_intro_singularities}

In 1965 Roger Penrose \cite{Pen65}, and later he and S. Hawking \cite{Haw66i,Haw66ii,Haw67iii,HP70,HE95}, proved a set of \textit{singularity theorems}. These theorems state that under reasonable conditions the spacetime turns out to be \textit{geodesic incomplete} -- \ie it has \textit{singularities}. Consequently, some researchers proclaimed that General Relativity predicts its own breakdown, by predicting the singularities \cite{HP70,Haw76,ASH91,HP96,Ash08,Ash09}. Hawking's discovery of the black hole evaporation, leading to his \textit{information loss paradox} \cite{Haw75,Haw76}, made the things even worse. The singularities seem to destroy information, in particular violating the unitary evolution of quantum systems. 
The reason is that the field equations cannot be continued through singularities.

By applying the results presented in this article we shall see that, at least for {\semireg} {\semiriem} manifolds, we can extend Einstein's equation through the singularities. Einstein's equation is replaced by a densitized version which is equivalent to the standard version if the metric is {\nondeg}. This equation remains smooth at singularities, which now become harmless.

\subsection{Einstein's equation on {\semireg} spacetimes}
\label{ss_einstein_tensor_densitized}

To define the Einstein tensor on a {\semireg} {\semiriem} manifold, we normally make use of the Ricci tensor and the scalar curvature:
\begin{equation}
\label{eq_einstein_tensor}
	G:=\ric-\frac 1 2 s g
\end{equation}
These two quantities can be defined even for a degenerate metric, so long as the metric doesn't change its signature (see \sref{s_ricci_tensor_scalar}), but at the points where the signature changes, they can become infinite.

\begin{definition}
\label{def_semi_reg_spacetime}
A \textit{{\semireg} spacetime} is a four{\hyph}dimensional {\semireg} {\semiriem} manifold having the signature $(0,3,1)$ at the points where it is {\nondeg}.
\end{definition}

\begin{theorem}
\label{thm_densitized_einstein}
Let $(M,g)$ be a {\semireg} spacetime. Then its Einstein density tensor of weight $2$, $G\det g$, is smooth.
\end{theorem}
\begin{proof}
At the points $p$ where the metric is {\nondeg}, the Einstein tensor  \eqref{eq_einstein_tensor} can be expressed using the Hodge $\ast$ operator by:
\begin{equation}
\label{eq_einstein_tensor_hodge}
	G_{ab} = g^{st}(\ast R\ast)_{asbt},
\end{equation}
where $(\ast R\ast)_{abcd}$ is obtained by taking the Hodge dual of $R_{abcd}$ with respect to the first and the second pairs of indices \cfeg{PeR87}{234}. Explicitly, if we write the components of the volume form associated to the metric as $\varepsilon_{abcd}$, we have
\begin{equation}
	(\ast R\ast)_{abcd} = \varepsilon_{ab}{}^{st}\varepsilon_{cd}{}^{pq}R_{stpq}.
\end{equation}
If we employ coordinates, the volume form can be expressed in terms of the Levi-Civita symbol by
\begin{equation}
\varepsilon_{abcd} = \epsilon_{abcd}\sqrt{-\det g}.
\end{equation}
We can rewrite the Einstein tensor as
\begin{equation}
\label{eq_einstein_tensor_hodge_lc}
	G^{ab} = \dsfrac{g_{kl}\epsilon^{akst}\epsilon^{blpq} R_{stpq}}{\det g},
\end{equation}

If we allow the metric to become degenerate, the Einstein tensor so defined becomes divergent, as it is expected. But the tensor density $G^{ab}\det g$, of weight $2$, associated to it remains smooth, and we get
\begin{equation}
\label{eq_einstein_tensor_density}
	G^{ab}\det g = g_{kl}\epsilon^{akst}\epsilon^{blpq} R_{stpq}.
\end{equation}

Since the spacetime is {\semireg}, this quantity is indeed smooth, because it is constructed only from the Riemann curvature tensor, which is smooth (see Theorem \ref{thm_riemann_curvature_semi_regular}), and from the Levi-Civita symbol, which is constant in the particular coordinate system. The determinant of the metric converges to $0$ so that it cancels the divergence which normally would appear in $G^{ab}$. The tensor density $G_{ab}\det g$, being obtained by lowering its indices, is also smooth.
\end{proof}

\begin{remark}
Because the densitized Einstein tensor $G_{ab}\det g$ is smooth, it follows that the densitized curvature scalar is smooth
\begin{equation}
\label{eq_curvature_scalar_density}
	s\det g = -g_{ab}G^{ab}\det g,
\end{equation}
and so is the densitized Ricci tensor
\begin{equation}
\label{eq_ricci_density}
	R_{ab}\det g = g_{as}g_{bt}G^{st}\det g + \dsfrac 1 2 s g_{ab}\det g.
\end{equation}
\end{remark}

\begin{remark}
In the context of General Relativity, on a {\semireg} spacetime, if $T$ is the stress-energy tensor, we can write the \textit{densitized Einstein equation}:
\begin{equation}
\label{eq_einstein:densitized}
	G\det g + \Lambda g\det g = \kappa T\det g,
\end{equation}
or, in coordinates or local frames,
\begin{equation}
\label{eq_einstein_idx:densitized}
	G_{ab}\det g + \Lambda g_{ab}\det g = \kappa T_{ab}\det g,
\end{equation}
where $\kappa:=\dsfrac{8\pi \mc G}{c^4}$, with $\mc G$ and $c$ being Newton's constant and the speed of light.
\end{remark}

\section{Applications}
\label{s_applications}

This paper introduces new mathematical tools to deal with singularities in {\semiriem} geometry, motivated mainly by the problems of singularities in General Relativity.

The mathematical tools introduced here were further developed in \cite{Sto11b}, where it is shown that the warped products provide a large class of {\semireg} manifolds. In \cite{Sto11d} we extend the Cartan structure equations to the degenerate case.

Other papers use these tools to deepen the understanding of cosmological singularities encountered in General Relativity.

The {\flrw} metrics are shown to be {\semireg} in \cite{Sto11h}. We also show that the densities $\rho\sqrt{-\det g}$,  $p\sqrt{-\det g}$,  and $T_{ab}\sqrt{-\det g}$ are smooth, and the densitized version of Einstein's equation \eqref{eq_einstein:densitized} holds (even with weight $1$).

The singularities studied here may seem to be too special, knowing that the metrics of the stationary black hole solutions have components which diverge while approaching the singularity. But for the {\schw} black holes, an appropriate coordinate transformation makes the metric analytic, and in fact {\semireg} \cite{Sto11e}. This can be viewed as analogous to the Eddington-Finkelstein coordinate transformations, which made the apparently singular metric on the event horizon of the {\schw} black hole become {\nondeg}. In  \cite{Sto11e} we do this for the genuine singularity at $r=0$, and of course we can't make it {\nondeg}, but we make it degenerate, analytical and {\semireg}. The {\rn} and {\kn} solutions can be made analytic at singularities too, and the electromagnetic potential and field become analytic too \cite{Sto11f,Sto11g}. The singularities of this type can be used to construct black holes which appear and disappear by evaporation, and they are compatible with global hyperbolicity \cite{Sto12e}.

A particular kind of {\semireg} singularities is introduced in \cite{Sto12b}, which admit smooth Ricci decomposition and allow the writing a tensorial form of Einstein's equation involving the Ricci part of the Riemann curvature, instead of the Ricci tensor. This class contains as subcases {\FLRW} singularities \cite{Sto12a}, $1+3$ degenerate warped product singularities, isotropic singularities, and the {\schw} solution \cite{Sto12b}. A big bang singularity of this kind satisfies automatically the Weyl curvature hypothesis \cite{Sto12c}. This hypothesis was proposed by Penrose to explain the second law of Thermodynamics, and of the high homogeneity and isotropy of the universe, especially around the Big Bang \cite{Pen79}.

All these applications are based on the methods introduced and developed in this article.

\textbf{Acknowledgements.}
Partially supported by Romanian Government grant PN II Idei 1187.



\end{document}